\documentclass[a4paper, reqno, 11pt]{amsart}

\usepackage[english]{babel}
\usepackage{amsmath}
\usepackage{mathrsfs}
\usepackage{amssymb}
\usepackage{amsthm}
\usepackage{enumerate}
\usepackage{ifthen}
\usepackage{bbm}
\usepackage{color}
\usepackage{xcolor}
\usepackage{graphicx}
\provideboolean{shownotes} 
\setboolean{shownotes}{true}
\usepackage{hyperref}
\usepackage{soul}
\usepackage{geometry}
\newcommand{\margnote}[1]{
\ifthenelse{\boolean{shownotes}}%
{\marginpar{\raggedright\tiny\texttt{#1}}}%
{}%
}

\newcommand{\hole}[1]{
\ifthenelse{\boolean{shownotes}}%
{\begin{center} \fbox{ \rule {.25cm}{0cm}
\rule[-.1cm]{0cm}{.4cm} \parbox{.85\textwidth}{\begin{center}
\texttt{#1}\end{center}} \rule {.25cm}{0cm}}\end{center}}
{}
}

\newtheorem{theorem}{Theorem}[section]   
\newtheorem{corollary}[theorem]{Corollary}
\newtheorem{lemma}[theorem]{Lemma}
\newtheorem{proposition}[theorem]{Proposition}

\theoremstyle{definition}
\newtheorem{definition}[theorem]{Definition}
\newtheorem{remark}[theorem]{Remark}

\numberwithin{equation}{section}


\begin{document}

\title[Stochastic Quantum-Navier-Stokes equations]{Weak martingale solutions to the stochastic 1D Quantum-Navier-Stokes equations}

\author[D. Donatelli, L. Pescatore, S. Spirito]{D. Donatelli, L. Pescatore, S. Spirito}
\address[D.Donatelli]{DISIM - Dipartimento di Ingegneria e Scienze dell'Informazione e Matematica\\ Universit\`a  degli Studi dell'Aquila \\Via Vetoio \\ 67100 L'Aquila \\ Italy}
\email[]{\href{donatella.donatelli@}{donatella.donatelli@univaq.it}}
\address[L.Pescatore]{DISIM - Dipartimento di Ingegneria e Scienze dell'Informazione e Matematica\\ Universit\`a  degli Studi dell'Aquila \\Via Vetoio \\ 67100 L'Aquila \\ Italy}
\email[]{\href{L.Pescatore@}{lorenzo.pescatore@graduate.univaq.it}}
\address[S. Spirito]{DISIM - Dipartimento di Ingegneria e Scienze dell'Informazione e Matematica\\ Universit\`a  degli Studi dell'Aquila \\Via Vetoio \\ 67100 L'Aquila \\ Italy}
\email[]{\href{stefano.spirito@}{stefano.spirito@univaq.it}}

\begin{abstract}
In this paper we prove the existence of global weak dissipative martingale solutions for a one-dimensional compressible fluid model with capillarity and density dependent viscosity,  driven by random initial data and a stochastic forcing term.  These solutions are weak in both PDEs and Probability sense and may have vacuum regions. The proof relies on the construction of an approximating system which provides extra dissipation properties and the  convergence is based on an appropriate truncation of the velocity field in the momentum equation and a stochastic compactness argument.
\end{abstract}

\maketitle

\section{Introduction}
In this paper we consider the one-dimensional Quantum-Navier-Stokes equations driven by random initial data and a stochastic force, given by the following set of equations
\begin{equation} \label{stoc quantum}
\begin{cases}
\text{d}\rho+\partial_x(\rho u)\text{d}t=0,\\
\text{d}(\rho u)+[ \partial_x(\rho u^2+p(\rho))]\text{d}t=[\partial_x(\mu(\rho)\partial_x u)+ \rho \partial_x \bigg( \dfrac{\partial_{xx} \sqrt{\rho}}{\sqrt{\rho}}\bigg) ]\text{d}t +\mathbb{G}(\rho,\rho u)\text{d}W,
\end{cases}
\end{equation}
where $x \in \mathbb{T},$ the one-dimensional flat torus, and $t \in (0,T).$ 
The  unknowns $\rho>0$ and $u\in\mathbb{R}$ represent the density and the velocity of the fluid, respectively.  Concerning the constitutive law for the pressure and the assumptions on the viscosity, we consider 
\begin{equation}\label{eq:alpha}
p(\rho)=\rho^{\gamma} , \quad \gamma > 1 , \quad \mu(\rho)= \rho^\alpha, \quad \frac{1}{2} <  \alpha \le 1,
\end{equation}
respectively. The stochastic forcing term $\mathbb{G}(\rho, u)\text{d}W$ is a multiplicative noise and the related stochastic integral is understood in the It$\hat{\text{o}}$ sense.
System \eqref{stoc quantum} is endowed with the following initial conditions 
\begin{equation} \label{C.I}
(\rho(x,t), u(x,t))_{|_{t=0}}=(\rho_0(x), u_0(x))
\end{equation} 
and periodic boundary conditions 
\begin{equation}\label{C.B}
\rho(0,t)=\rho(1,t), \quad u(0,t)=u(1,t).
\end{equation}
The Quantum-Navier-Stokes equations fall within the wider class of fluid dynamical evolution equations called the Navier-Stokes-Korteweg equations,  where capillarity effects are not negligible and are indeed included in the dynamics of the fluid. In particular,  the multi-dimensional case reads as follows
\begin{equation} \label{3D}
\begin{cases}
\partial_t \rho + \operatorname{div}(\rho u)=0, \\
\partial_t (\rho u) + \operatorname{div} (\rho u \otimes u)+ \nabla p = \operatorname{div} \mathbb{S} + \operatorname{div}\mathbb{K},
\end{cases}
\end{equation}
where $\mathbb{S}=\mathbb{S}(\nabla u)$ is the viscosity stress tensor $$ \mathbb{S}= \mu(\rho) D u+ g(\rho) \operatorname{div}u \mathbb{I},$$
with $\mu(\rho), g(\rho)$ viscosity coefficients satisfying the Lam\'e relation $\mu(\rho) \ge 0, \; \mu(\rho)+d g(\rho) \ge 0$ and  $\mathbb{K}=\mathbb{K}(\rho, \nabla \rho)$ is the capillarity tensor given by 
$$ \mathbb{K}= \bigg ( \rho \operatorname{div}(k(\rho) \nabla \rho)-\dfrac{1}{2}(\rho k'(\rho)-k(\rho)) | \nabla \rho |^2 \bigg ) \mathbb{I}-k(\rho) \nabla \rho \otimes \nabla \rho,$$
with $k(\rho)$ denoting the capillary coefficient.  
\\
The formal derivation of the Korteweg tensor $\mathbb{K}$ has been performed in \cite{Dunn}.  The choice of the capillarity coefficient $k(\rho)=\frac{1}{\rho}$ reduces the Korteweg tensor to the so-called Bohm potential $\frac{\Delta{\sqrt{\rho}}}{\sqrt{\rho}}$, which in the inviscid case appears in the description of Quantum Hydrodynamics model (QHD) for superfluids and Bose-Einstein condensate, see \cite{Landau}. The name Quantum-Navier-Stokes equations (QNS) arose in the literature by analogy with the inviscid case.
The system \eqref{stoc quantum} is the one-dimensional stochastic version of the Navier-Stokes-Korteweg equations \eqref{3D} with the choices $g(\rho)=0,  \;k(\rho)=\frac{1}{\rho}$ and stochastic forcing term $\mathbb{G}(\rho, u)\text{d}W.$ 
\\ \\
In this paper we prove the existence of global weak dissipative martingale solutions for the system \eqref{stoc quantum}-\eqref{C.B}. These solutions are weak in both PDEs and Probability sense and vacuum regions of the density are allowed.  To the best of our knowledge,  this is the first result concerning the global existence of solutions for the Stochastic-Quantum-Navier-Stokes equations in the framework of weak solutions. As, a byproduct of our result we also obtain the global existence of weak solutions for the deterministic case $\mathbb{G}=0.$  In particular, we extend the theory of weak solutions developed in the monograph \cite{Feir} for stochastically forced compressible flows to the case of the Quantum-Navier-Stokes equations.  Such an extension requires non-trivial arguments due to the presence of the Korteweg tensor and its non-linear structure.
In particular, we adapt to the stochastic framework to the techniques introduced in \cite{Lacroix} and used in \cite{Spirito, Ant6, Bresch}.\\

In our analysis, vacuum regions are considered in the definition of weak dissipative martingale solutions and in order to prove our existence result we make use of an approximating system which enjoys extra dissipation properties and admits more regular solutions. In particular, we note that the velocity field $u$ and its derivative $\partial_x\,u$ may be not uniquely defined on the set $\{ \rho=0 \}$ and then weak solutions must be properly defined. In order to do that, as common in this case, the definition of weak solutions in given in terms of $\Lambda,  \zeta$ which satisfy the following identities
$$ \sqrt{\rho}\Lambda=m, \quad \rho^{\frac{\alpha}{2}}\zeta= \partial_x(\rho^{\alpha-\frac{1}{2}}\Lambda)-2\rho^{\frac{\alpha-1}{2}}\Lambda \partial_x \rho^{\frac{\alpha}{2}}.$$
As a consequence, the appropriate application of our stochastic compactness argument, as well as the identification of the limit process as a weak dissipative martingale solutions, presents some technical difficulties. Nevertheless we stress that stochastic effects are considered in both random initial data and in the stochastic forcing term $\mathbb{G}$ for which we  assume only a limited amount of regularity, see Remark \ref{Remark 2}.\\

Next we comment of the range of the exponents of the viscosity coefficients. A main tool in our analysis is the so-called BD Entropy introduced in \cite{BD}, namely an {\em a priori} estimate providing additional regularity on the density. In particular, the BD Entropy implies a control on $\rho^{\alpha-\frac{1}{2}}\in L^{\infty}_t(H_x^{1})$ and then for $\alpha \in [0,\frac{1}{2}],$ vacuum regions cannot appear, see \cite{D.P.S.}. In this paper we consider the range $\alpha \in (\frac{1}{2}, 1]$, so vacuum regions may appear and thus only existence of weak solutions is expected. 
In general, if we consider the $1D$ Navier-Stokes-Korteweg equations with viscosity coefficient $\mu(\rho)=\rho^{\alpha}$ and capillarity $\kappa(\rho)=\rho^{\beta}$ in \cite{Le Floch} almost optimal conditions for the validity of the BD Entropy are given. In particular, for the (QNS) the range is $\alpha\in[0,\frac{3}{2}]$. In this paper we assume $\alpha\in(\frac{1}{2},1]$, which corresponds to the condition of {\em Tame Capillarity} introduced in \cite{Le Floch}. This condition generalizes  the common condition in the theory of hyperbolic conservation laws that capillarity is dominated by the square of the viscosity to the case of solutions with vacuum, we refer to \cite{Le Floch} for more insights. In particular, the viscous effects are dominant in the dynamic and thus, together with the results in \cite{D.P.S.} where the range $\alpha\in[0,\frac{1}{2}]$ is considered, the present paper covers all the possible viscosity exponents for which the global existence can be obtained with the present methods.\\


Finally, we describe the state of art concerning the (QNS) system. For the existence results for (QNS) in the framework of weak solutions we refer to  \cite{Spirito2}, \cite{Lacroix} and to \cite{Spirito} for the case of Navier-Stokes-Korteweg with constant capillarity. We also recall the result in \cite{Jung qns} where a different notion of weak solutions is used. We also refer to \cite{Ant6} for the case with non-trivial far field behavior and to \cite{Burtea1} and \cite{Chen} for the one-dimensional case. Global existence of weak solutions has also been proved in the inviscid case,  in particular we mention \cite{Ant2}, \cite{Ant1}, \cite{Ant3}, \cite{Ant4} for the global existence of finite energy weak solutions and related results for the (QHD) system, and, for the general Euler-Korteweg system, we refer \cite{Benz} for the local existence of smooth solutions, to \cite{Aud} for small data well-posedness, and \cite{Feir. Don.} for some non-uniqueness results using convex integration. In the case of zero-capillarity coefficient, which corresponds to the compressible Navier-Stokes equations with degenerate viscosity, we recall that the global existence of weak solutions in three and two dimensions has been proved in \cite{Li Xin} and \cite{Vasseur Yu 0} for the case of linear viscosity and in \cite{Bresch} for some non-linear viscosity of power law type. Moreover, for the one-dimensional case concerning the global regularity of solutions in the case of degenerate viscosity several results are available, see \cite{Mellet, JX, Const, Const2, Haspot 2, Burtea2, Charve}.
 We stress however that none of these results concerns stochastic flows except of \cite{D.P.S.}.  Indeed regarding stochastically forced compressible fluid flows, most of the literature deal with the compressible Navier-Stokes equations in the case of constant viscosity \cite{Feir}, besides the results in \cite{Zatorska} where the authors proved sequential stability of weak martingale solutions for the compressible Navier-Stokes equations with degenerate linear viscosity.\\
 
 Finally, also some singular limits results has been obtained for the Quantum-Navier-Stokes equations and Korteweg fluids.  In particular we refer to \cite{Ant5} and \cite{Donatelli} for the low and high Mach number limit respectively,  in the case of linear density dependent viscosity and we mention \cite{Ant0} for high friction limit from the Quantum-Navier-Stokes equations to the Quantum drift-diffusion equations. \\

The outline of the paper is the following: in Section 2 we introduce the stochastic elements of our analysis by discussing the assumptions on the stochastic forcing term and on the random initial data.  We also recall some useful results concerning the stochastic calculus for SPDEs and stochastic compactness tools.  Section 3 is dedicated to the statement of the definition of weak dissipative martingale solutions and of our main result, Theorem \ref{main theorem weak}.  In Section 4 we introduce the approximating system for which we prove the global well-posedness in the framework of strong pathwise solutions which are strong solutions in both PDEs and Probability sense.  Section 5 is dedicated to the proof of Theorem \ref{main theorem weak} which is performed by introducing a suitable truncation of the velocity field $u$ in the momentum equations and by proving the convergence result by using a stochastic compactness argument.
\section{Preliminaries}
In this Section we fix the notation, we give a detailed description of the stochastic setting and we state the regularity assumptions on the stochastic forcing term and initial conditions.  We also recall some basic results in the framework of stochastic partial differential equations and stochastic compactness analysis.  The presented stochastic theory has been discussed in a very detailed way in the monograph \cite{Feir}, where the authors give an exhaustive overview of the stochastic elements needed to deal with stochastically forced compressible fluid flows.  For the reader's convenience we recall here the main tools used in our analysis.
\\
\subsection{Notation} 

In the sequel we assume the following notations and conventions:
\begin{enumerate}
\item
Let $\mathbb{T}$ be the one-dimensional flat torus,  we denote by $L^p(\mathbb{T})$ the standard Lebesgue spaces  and with $\| \cdot \|_{L^p}$ their norm. $W^{k,p}(\mathbb{T})$ denotes the Sobolev space of $L^p(\mathbb{T})$ functions with $k$ distributional derivatives in $L^p(\mathbb{T})$ and $H^k(\mathbb{T})$ corresponds to the case $p=2.$ For a given Banach space $X$ we consider the Bochner space for time dependent functions with values in $X$,  namely $C(0,T;X)$, $L^p(0,T;X).$
\item
Given two functions $F, G$ and a variable $p,$ we denote by $F \lesssim G$ and $F \lesssim_{p} G$ the existence of a generic constant $c>0$ and $c(p)>0$  such that $F \le cG$ and $F \le c(p)G$ respectively.
\item
Let $f$ be a given random variable,  we denote by $\mathfrak{L}[f]$ the law of $f$ and by $\mathfrak{L}_X[f]$ the law of $f$ on the space $X$.  $L^p_{\text{prog}}(\Omega \times [0,T])$ represents the Lebesgue space of functions that are measurable with respect to the $\sigma$-field of $(\mathfrak{F}_t)$-progressively measurable sets in $\Omega \times [0,T].$ For a  given stochastic process $U,$ $(\sigma_t [U])_{t \ge 0}$ represents the canonical filtration of $U.$
\end{enumerate}
\subsection{The stochastic setting and stochastic compactness tools}
Given $(\Omega, \mathfrak{F}, (\mathfrak{F}_t)_{t \ge 0}, \mathbb{P})$ a stochastic basis with right continuous filtration,  the initial data $(\rho_0, u_0)$ are $\mathfrak{F}_0$-measurable random variables ranging in a suitable regularity space.
The stochastic process $W$ is a cylindrical $(\mathfrak{F}_t)$-Wiener process in a separable Hilbert space $\mathfrak{U},$ namely 
\begin{equation}\label{Cyl W}
W(t)= \sum_{k=1}^{\infty} e_kW_k(t),
\end{equation} where $(W_k)_{k \in \mathbb{N}}$ is a sequence of mutually independent real valued Wiener processes relative to $(\mathfrak{F}_t)_{t \ge 0}$ and $(e_k)_{k \in \mathbb{N}}$ is an orthonormal basis of $\mathfrak{U}.$
The sum in \eqref{Cyl W} is not a priori converging in $\mathfrak{U}$,  hence we construct a larger space $\mathfrak{U}_0 \supset \mathfrak{U}$ as follows
\begin{equation}
\mathfrak{U}_0= \bigg\{ v= \sum_{k \ge 1}^{} c_k e_k; \quad \sum_{k \ge 1}^{} \dfrac{c^2_k}{k^2} < \infty \bigg\},
\end{equation}
endowed with the norm
\begin{equation*}
\| v \|^2_{\mathfrak{U}_0}= \sum_{k \ge 1}^{} \dfrac{c^2_k}{k^2}, \quad v= \sum_{k \ge 1}^{} c_k e_k.
\end{equation*}
The diffusion operator is defined by superposition $$\mathbb{G}(\rho, \rho u): \mathfrak{U} \rightarrow L^1(\mathbb{T}), \quad \mathbb{G}(\rho, \rho u )e_k= G_k(\cdot, \rho(\cdot), \rho u(\cdot)).$$
where the coefficients  $$G_k: \mathbb{T} \times [0, \infty) \times \mathbb{R} \rightarrow \mathbb{R}$$ are $C^s$-functions, for $s \in \mathbb{N}$ specified below, which satisfy the following growth assumptions uniformly in $x \in \mathbb{T}$ 
\begin{equation}\label{G1}
G_k(\cdot,0,0)=0,
\end{equation}
\begin{equation}\label{G w}
| G_k(x,\rho,q) | \le g_k (\rho+|q|),
\end{equation}
\begin{equation}\label{G2}
| \partial^l_{x,\rho,q} G_k(x, \rho,q) | \le g_k,  \quad \sum_{k=1}^\infty g_k < \infty \quad  \text{for all} \; l\in {1,...,s}.
\end{equation}
In the framework of weak solutions the momentum equations is satisfied only in the distributional sense.  Consequently,  by virtue of the Sobolev embedding $L^1(\mathbb{T}) \hookrightarrow W^{-l,2}(\mathbb{T}),$ provided that $l > \frac{1}{2},$ the stochastic integral has to be understood as a stochastic process in the Hilbert space $W^{-l,2}(\mathbb{T}).$ To be precise $\mathbb{G}(\rho,\rho u) \in L_2(\mathfrak{U}; W^{-l.2}(\mathbb{T}))$ the space of Hilbert-Schmidt operators form $\mathfrak{U}$ to $W^{-l.2}(\mathbb{T}).$ Indeed we have 
\begin{equation}\label{G W-}
\begin{split}
& \| \mathbb{G}(\rho, \rho u ) \|^2_{L_2(\mathfrak{U}; W_x^{-l,2})}= \sum_{k=1}^{\infty} \| G_k(\rho, \rho u ) \|^2_{W_x^{-l,2}} \le C \sum_{k=1}^{\infty} \| G_k( \rho, \rho u ) \|^2_{L^1_x} \\ & \lesssim \int_{\mathbb{T}} \bigg( \sum_{k=1}^{\infty} \rho^{-1} | G_k(x, \rho, \rho u ) |^2 \bigg) dx \lesssim \int_{\mathbb{T}} ( \rho+ \rho |u|^2) dx.
\end{split}
\end{equation}
Furthermore we observe that the right hand side of \eqref{G W-} is finite whenever we are dealing with finite energy solutions.  To be precise, if 
\begin{equation}\label{reg int stoch mart}
\begin{split}
& \rho \in L^\gamma_{prog}(\Omega \times (0,T); L^\gamma(\mathbb{T})), \quad \gamma \ge 1, \\ & \sqrt{\rho}u \in L^2_{prog}(\Omega \times (0,T); L^2(\mathbb{T})),
\end{split}
\end{equation}
then the stochastic integral is a well-defined $(\mathfrak{F}_t)$-martingale taking values in $W^{-l,2}(\mathbb{T}).$ Moreover we point out that the regularity \eqref{reg int stoch mart} can be weakened as follows $$ \rho \in L^\gamma(0,T); L^\gamma(\mathbb{T})) \quad \mathbb{P}-a.s. ,\quad  \sqrt{\rho}u \in L^2 (0,T; L^2(\mathbb{T})) \quad \mathbb{P}-a.s. $$
and in this case the stochastic integral is in general only a local martingale.
We refer the reader to \cite{Da Prato} for a complete overview on the stochastic setting and to \cite{Feir} for a detailed description of the stochastic elements used in order to study compressible fluid dynamical equations.
\\
\\
In order to be consistent with the deterministic case where the external force is usually of the form $G=\rho f,$ we will consider stochastic forcing terms of the form $G_k(x, \rho, \rho u )= \rho F_k(x,\rho, u)$ with regularity of $F_k$ determined by \eqref{G1}-\eqref{G2}.  Precisely $$F_k: \mathbb{T} \times [0,\infty) \times \mathbb{R} \rightarrow \mathbb{R}, \quad F_k \in C^s(\mathbb{T} \times (0,\infty) \times \mathbb{R})$$ satisfy
\begin{equation}\label{f1}
F_k(\cdot,0,0)=0,
\end{equation}
\begin{equation}\label{f2}
| \partial^l_{x,\rho,q} F_k(x, \rho,q) | \le \alpha_k,  \quad \sum_{k=1}^\infty \alpha_k < \infty \quad  \text{for all} \; l\in {1,...,s}.
\end{equation}
Our setting includes also the case of an additive noise $\sigma(x)\text{d}W.$ This leads to the study of a wide class of problems related to the existence and uniqueness of invariant measures which has been recently investigated for the compressible Navier-Stokes equation,  see \cite{Coti}. \\
We point out that random effects are considered also in the initial data.  The initial conditions $(\rho_0,u_0)$ are indeed $\mathfrak{F}_0$-measurable $H^{s+1}(\mathbb{T})\times H^{s}(\mathbb{T})$ valued random variables, with $s>7/2$, such that there exists a deterministic constant $C>0$ such that 
\begin{equation} \label{C.I STRONG}
 C^{-1} \le \rho_0 \le C,  \quad \rho_0 \in L^p(\Omega; H^{s+1}(\mathbb{T})),  \quad u_0 \in L^p(\Omega; H^s(\mathbb{T})),  \quad \text{for all} \; p \in [1,\infty).
\end{equation}
In the sequel we introduce some useful instruments which are needed in order to deal with the stochastic elements of our analysis.  We start by recalling two results concerning the It$\hat{\text{o}}$ stochastic calculus.
\begin{theorem}\label{Ito Lemma}
Let $W$ be an $\mathfrak{F}_t$-cylindrical Wiener process on the stochastic basis \newline $(\Omega, \mathfrak{F},(\mathfrak{F}_t
)_{t \ge 0}, \mathbb{P}).$
Let $(r,s)$ be a pair of stochastic processes on $(\Omega, \mathfrak{F},(\mathfrak{F}_t
)_{t \ge 0}, \mathbb{P})$ satisfying
\begin{equation}
\text{d}r=[Dr]\text{d}t+[\mathbb{D}r] \text{d}W, \quad \quad \text{d}s=[Ds]\text{d}t+[\mathbb{D}s] \text{d}W
\end{equation}
on the cylinder $(0,T) \times \mathbb{T}^d.$ Now suppose that the following 
\begin{equation}
r \in C^{\infty} ([0,T] \times \mathbb{T}^d), \quad \quad s \in C^{\infty} ([0,T] \times \mathbb{T}^d)
\end{equation}
holds $\mathbb{P}$-a.s. and that for all $1 \le q < \infty$
\begin{equation}
\mathbb{E} \bigg[ \sup_{t \in [0,T]} \| r \|^2_{W^{1,q}_x} \bigg]^q+\mathbb{E} \bigg[ \sup_{t \in [0,T]} \| s \|^2_{W^{1,q}_x} \bigg]^q \lesssim_q 1
\end{equation}
Furthermore,  assume that $[Dr], \; [Ds], \; [\mathbb{D}r], \; [\mathbb{D}s]$ are progressively measurable and that
\begin{equation}
\begin{split}
& [Dr], \; [Ds], \;  \in L^q(\Omega; L^q(0,T;W^{1,q}_x) \\ &
[\mathbb{D}r], \; [\mathbb{D}s] \in L^2( \Omega; L^2(0,T;L_2(\mathfrak{U};L^2_x) )
\end{split}
\end{equation}
and 
\begin{equation}
\bigg( \sum_{k \in \mathbb{N}} |[Dr](e_k)|^q\bigg)^{\frac{1}{q}}, \;\bigg( \sum_{k \in \mathbb{N}} |[Ds](e_k)|^q\bigg)^{\frac{1}{q}} \in L^q(\Omega \times (0,T) \times \mathbb{T}^d) 
\end{equation}
holds. Finally,  for some $\lambda \ge 0.$ let $Q$ be $(\lambda +2)$-continuously differentiable function such that 
\begin{equation}
\mathbb{E} \sup_{t \in [0,T]} \| Q^j(r) \|^2_{W^{\lambda,q'} \cap C_x } < \infty, \quad j=0,1,2.
\end{equation}
Then 
\begin{equation}
\begin{split}
& \int_{\mathbb{T}^d} (sQ(r))(t) dx= \int_{\mathbb{T}^d} (s_0Q(r_0))(t) dx+\int_{0}^{t} \int_{\mathbb{T}^d} \bigg[ sQ'(r)[Dr]+\frac{1}{2} \sum_{k \in \mathbb{N}} sQ''(r) | [\mathbb{D}r](e_k)|^2 \bigg] dxdt' \\ & +\int_{0}^{t} \int_{\mathbb{T}^d} Q(r)[Ds] dxdt'+\sum_{k\in\mathbb{N}} \int_{0}^{t} \int_{\mathbb{T}^d} [\mathbb{D}s](e_k)[\mathbb{D}r](e_k) dxdt' \\ & + \sum_{k\in \mathbb{N}} \int_{0}^{t} \int_{\mathbb{T}^d} \bigg[ sQ'(r)[\mathbb{D}r](e_k)+Q(r)[\mathbb{D}s](e_k)\bigg] dx dW_k(t').
\end{split}
\end{equation}
\end{theorem}
\noindent
\begin{theorem}(It$\hat{\text{o}}$'s formula) \\
Let $W$ be a cylindrical Wiener process in a separable Hilbert space $\mathfrak{U}.$ Let $\mathbb{G}$ be an $L_{2}(\mathfrak{U},H)$-valued stochastically integrable process, let $g$ be an $H$-valued progressively measurable Bochner integrable process, and let $U(0)$ be a $\mathfrak{F}_0$-measurable $H$-valued random variable so that the process $$ U(t)= U(0)+ \int_{0}^{t} g(s)ds +\int_{0}^{t} \mathbb{G}(s) dW(s)$$ is well defined.
Let $F: \, H \longrightarrow \mathbb{R}$ be a function such that $F, \, F',\, F''$ are uniformly continuous on bounded subset of $H.$ Then the following holds $\mathbb{P}$-a.s. for all $t \in [0,T]$
\begin{equation}
\begin{split}
& F(U(t))=F(U(0))+ \int_{0}^{t} \langle F'(U(s)),g(s) \rangle ds+ \int_{0}^{t} \langle F'(U(s)),G(s) dW(s) \rangle \\ & 
+\dfrac{1}{2} \int_{0}^{t} \text{Tr}(G(s)^*F{''}(U(s))G(s))ds
\end{split}
\end{equation}
where $\text{Tr} A= \sum_{k=1}^{\infty} \langle A e_k,e_k \rangle$ for $A$ being a bounded linear operator on H.
\end{theorem}
The next theorem is an important tool needed in order to perform a stochastic compactness argument.  We introduce it in the case of sub-Polish spaces which is appropriate when dealing weak convergences and low regularity as in the case of weak solutions.
\begin{theorem}(Jakubowski-Skorokhod representation theorem)\label{Jak}.
Let $(X, \tau)$ be a sub-Polish space and let $\mathcal{S}$ be the $\sigma-$ field generated by $\{f_n: n \in \mathbb{N} \}.$ If $(\mu_n)_{n \in \mathbb{N}}$ is a tight sequence of probability measures on $(X, \tau),$ then there exists a subsequence $(n_k)$ and $X-$ valued Borel measurable random variables $(U_k)_{k \in \mathbb{N}}$ and $U$ defined on the standard probability space $([0.1], \overline{{\mathcal{B}}([0,1])}, \mathcal{L}),$ such that $\mu_{n_k}$ is the law of $U_k$ and $U_k(\omega)$ converges to $U(\omega)$ in $X$ for every $\omega \in [0,1].$ Moreover, the law of $U$ is a Radon measure.
\end{theorem}
The following Proposition is a key result used in order to estimate stochastic integrals.  
\begin{proposition}(Burkholder-Davis-Gundy's inequality) \label{BDG}\\
Let $H$ be a separable Hilbert space.  Let $p \in (0, \infty).$ There exists a constant $C_p>0$ such that,  for every $(\mathfrak{F}_t)$-progressively measurable stochastic process $\mathbb{G}$ satisfying 
\begin{equation}
\mathbb{E} \int_{0}^{T} \| \mathbb{G}(t) \|^2_{L_{2}(\mathfrak{U},H)}dt \ < \infty,
\end{equation}
the following holds:
\begin{equation}
\mathbb{E} \sup_{ t \in [0,T]} \bigg{ \| } \int_{0}^{t} \mathbb{G}(s)dW(s) \bigg{ \| }^p_H \le C_p \mathbb{E} \bigg( \int_{0}^{T} \| \mathbb{G}(s) \|^2_{L_{2}(\mathfrak{U},H)} ds \bigg)^{\frac{p}{2}}.
\end{equation}
\end{proposition}
We conclude this section by recalling two results concerning the equality in law of solutions of stochastic partial differential equations and the convergence of the stochastic integral.  These results are useful in order to identify the limit process of our approximation technique as a weak solutions of \eqref{stoc quantum}-\eqref{C.I}.
\begin{theorem}\label{eq in law equation}
Let $U \in L^1_{loc}(\mathbb{R}; L^1(\mathbb{T}^d)) \; \mathbb{P}-a.s.,$ $U(t)=U_0$ for $t \le 0,$ be a random distribution satisfying 
\begin{equation}
\text{d}D(U)+ \operatorname{div}F(U)\text{d}t=G(U)\text{d}W, \quad D(U_0)=D_0
\end{equation}
with the following regularity assumptions
\begin{equation}
D(U), \; F(U), G_k(U) \in L^1_{loc}(\mathbb{R}; L^1(\mathbb{T}^d)), \; \mathbb{P}-a.s.
\end{equation}
and
\begin{equation}
\int_{0}^{T} \sum_{k=1}^{\infty} | \langle G_k(U), \psi \rangle |^2 dt < \infty, \quad \mathbb{P}-a.s.
\end{equation}
meaning $(\langle G_k(U), \psi \rangle )_{k \in \mathbb{N}} \in L^2(0,T;L_2(\mathfrak{U}; \mathbb{R})), \; \mathbb{P}-a.s. \; \text{for any} \; \psi \in C^\infty(\mathbb{T}^d).$ \\
Let $W=(W_k)_{k\in \mathbb{N}}$ be a cylindrical Wiener process.  Suppose that the filtration $$\mathfrak{F}_t= \sigma\bigg( \sigma_t[U] \cup \bigcup_{k=1}^{\infty} \sigma_t[W_k] \bigg), \; t \ge 0,$$ is non-anticipative with respect to $W.$ Let $\tilde{U}$ be another random distribution and $\tilde{W}$ another stochastic process such that their joint laws coincide,  namely 
$$\mathcal{L}[U,W]=\mathcal{L}[\tilde{U},\tilde{W}] \quad \text{or} \quad [U,W] \overset{d}{\sim}[\tilde{U}, \tilde{W}].$$ Then $\tilde{W}$ is a cylindrical Wiener process, the filtration $$\tilde{\mathfrak{F}}_t= \sigma\bigg( \sigma_t[\tilde{U}] \cup \bigcup_{k=1}^{\infty} \sigma_t[\tilde{W}_k] \bigg), \; t \ge 0,$$ is non-anticipative with respect to $\tilde{W},$ and 
\begin{equation}
\begin{split}
& \mathcal{L}_{\mathbb{R}} \bigg[ \int_{0}^{T} [ \partial_t \psi \langle D(U), \psi \rangle + \psi \langle F(U), \nabla \psi \rangle ] dt + \int_{0}^{T} \psi \langle G(U), \psi \rangle dW + \psi(0) \langle D(U_0) \psi \rangle \bigg]= \\ & \mathcal{L}_{\mathbb{R}} \bigg[ \int_{0}^{T} [ \partial_t \psi \langle D(\tilde{U}), \psi \rangle + \psi \langle F(\tilde{U}), \nabla \psi \rangle ] dt + \int_{0}^{T} \psi \langle G(\tilde{U}), \psi \rangle dW + \psi(0) \langle D(\tilde{U}_0) \psi \rangle \bigg]
\end{split}
\end{equation}
for any deterministic $\psi \in C^{\infty}_c ([0,T)\times \mathbb{T}^d).$
\end{theorem}
\begin{lemma}\label{Stoch int conv Lemma}
Let $(\Omega, \mathfrak{F}, \mathbb{P})$ be a complete probability space.  For $n \in \mathbb{N},$ let $W_n$ be an $(\mathfrak{F}^n_t)-$cylindrical Wiener process and let $G_n$ be an $(\mathfrak{F}^n_t)-$ progressively measurable stochastic process such that $G_n \in L^2(0,T; L_2(\mathfrak{U}; W^{l,2}(\mathbb{T}^d))) \; a.s.$ Suppose that $$ W_n \rightarrow W \quad \text{in} \; C([0,T]; \mathfrak{U}_0) \; \text{in probability},$$ $$ G_n \rightarrow G \quad \text{in} \; L^2(0,T; L_2(\mathfrak{U}; W^{l,2}(\mathbb{T}^d))) \; \text{in probability},$$
where $W= \sum_{k=1}^{\infty} e_k W_k.$ Let $(\mathfrak{F}_t)_{t \ge 0}$ be the filtration given by $$\mathfrak{F}_t= \sigma \bigg( \bigcup_{k=1}^{\infty} \sigma_t [Ge_k] \cup \sigma_t [W_k] \bigg).$$
Then after a possible change on a set of zero measure in $\Omega \times (0,T),$ $G$ is $(\mathfrak{F}_t)-$ progressively measurable and $$ \int_{0}^{\cdot} G_n d W_n \rightarrow \int_{0}^{\cdot} G dW \quad \text{in} \; L^2(0,T; L_2(\mathfrak{U}; W^{l,2}(\mathbb{T}^d))) \; \text{in probability}.$$
\end{lemma}
\section{Main results}
Our analysis is focused on the existence of dissipative martingale solutions i.e. weak solutions in both PDEs and probability sense.  This means that the probability space and the driving Wiener process are not specified in advance and the stochastic elements are part of the solution itself.  Moreover, an additional difficulty, when dealing with weak solutions in the context of stochastic process, relies on the fact that the a priori estimates usually hold for a.e. $t \in (0,T)$ hence the unknowns cannot be interpreted as standard stochastic processes and therefore it is more convenient to consider these fields as random distributions.  This motivates the following definition:
\begin{definition}(Weak dissipative martingale solution) \label{def weak dissip sol} 
\\
Let $\nu=\nu(\rho,\rho u)$ be a Borel probability measure on $L^1(\mathbb{T}) \times L^1(\mathbb{T})$ s.t.  
\begin{equation}
\nu \{ \rho \ge0\}=1, \quad  \int_{L^1_x \times L^1_x} \bigg| \int_{\mathbb{T}} \bigg[\dfrac{1}{2}\rho |u|^2+ \dfrac{\rho^\gamma}{\gamma-1}+ | \partial_x \sqrt{\rho}|^2 dx \bigg|^r d\nu(\rho,\rho u ) < \infty, \quad r\ge 1.
\end{equation}
The quantity $((\Omega, \mathfrak{F}, (\mathfrak{F}_t)_{t \ge 0}, \mathbb{P}), \rho, \Lambda, \zeta)$ is called a weak dissipative martingale solution to system \eqref{stoc quantum}-\eqref{C.I} with initial law $\nu$ if 
\\
\begin{itemize}
\item[(1)]
$(\Omega, \mathfrak{F}, (\mathfrak{F}_t)_{t \ge 0}, \mathbb{P})$ is a stochastic basis with right continuous filtration;
\\
\item[(2)] 
$W$ is a cylindrical $(\mathfrak{F}_t)$-Wiener process;
\\
\item[(3)] 
$\rho, \Lambda, \zeta$ are random distributions adapted to $(\mathfrak{F}_t)_{t \ge 0},\; \rho \ge 0 \quad \mathbb{P}$-a.s.;
\\
\item[(4)]
There exists an $\mathfrak{F}_0$-measurable random variable $[\rho_0,u_0]$ s.t. $\nu= \mathcal{L}[\rho_0,u_0];$ 
\\
\item[(5)] 
Integrability conditions:
\begin{equation}
\begin{split}
& \rho \in L^\infty(0,T;H^1(\mathbb{T}) \cap L^2(0,T;H^2(\mathbb{T})),  \quad \sqrt{\rho}u \in L^\infty(0,T;L^2(\mathbb{T})),  \\ & \rho^{\frac{\gamma}{2}} \in L^\infty(0,T;L^2(\mathbb{T})) \cap L^2(0,T;H^1(\mathbb{T})),  \quad \partial_x \sqrt{\rho} \in L^\infty(0,T;L^2(\mathbb{T})), \\ &
\zeta \in L^2(0,T;L^2(\mathbb{T})),  \quad \partial_{xx} \rho^{\frac{\alpha}{2}} \in L^2(0,T;L^2(\mathbb{T})), \quad \partial_x \rho^{\frac{\alpha}{4}} \in L^4(0,T;L^4(\mathbb{T})), \quad  \mathbb{P}-a.s.
\end{split}
\end{equation}
\\
\item[(6)] The continuity equation
\begin{equation}
\int_{\mathbb{T}} \rho_0(x)\psi(0,x)dx+\int_{\mathbb{T}}\int_{0}^{T} \rho \partial_t \psi dxdt +\int_{\mathbb{T}}\int_{0}^{T} m \partial_x \psi dxdt =0,  
\end{equation}
$\text{holds for all} \; \psi \in C^\infty_c ( (0,T) \times \mathbb{T};\mathbb{R}),\quad \mathbb{P}-a.s.$
\\
\item[(7)] The momentum equation
\begin{equation}
\begin{split}
& \int_{\mathbb{T}} \rho_0 u_0 \psi(0,x)dx+ \int_{0}^{T} \int_{\mathbb{T}} \rho u \partial_t \psi dxdt+ \int_{0}^{T} \int_{\mathbb{T}} \rho u^2 \partial_x \psi dxdt+ \int_{0}^{T} \int_{\mathbb{T}} \rho^{\gamma} \partial_x \psi dxdt \\ &  -\int_{0}^{T} \int_{\mathbb{T}} \rho^{\frac{\alpha}{2}} \zeta \partial_x \psi dxdt- \dfrac{1}{2} \int_{0}^{T} \int_{\mathbb{T}}  \rho \partial_{xxx} \psi -2 \int_{0}^{T} \int_{\mathbb{T}}| \partial_x \sqrt{\rho} |^2
 \psi dxdt + \int_{0}^{T} \mathbb{G}(\rho, \rho u )\psi \text{d}W =0,  
\end{split}
\end{equation}
$\text{holds for all} \; \psi \in C^\infty_c ( (0,T) \times \mathbb{T};\mathbb{R}),\quad \mathbb{P}-a.s.$
\\
\item[(8)]The energy dissipation
\begin{equation}\label{energy dissipation equality}
\int_{0}^{T} \int_{\mathbb{T}} \rho^{\frac{\alpha}{2}} \zeta \psi dxdt= - \int_{0}^{T} \int_{\mathbb{T}} \rho^{\alpha-\frac{1}{2}} \Lambda \partial_x \psi dxdt -2\int_{0}^{T} \int_{\mathbb{T}} \rho^{\frac{\alpha-1}{2}} \Lambda \partial_x \rho^{\frac{\alpha}{2}} \psi dxdt, 
\end{equation}
$\text{holds for all} \; \psi \in C^\infty_c ( (0,T) \times \mathbb{T};\mathbb{R}),\quad \mathbb{P}-a.s.$
\\
\item[(9)] 
The energy inequality
\begin{equation}\label{energy weak def}
\begin{split}
&  -\int_{0}^{T} \partial_t \phi  \int_{\mathbb{T}} \bigg[\dfrac{1}{2}\rho |u|^2+ \dfrac{\rho^\gamma}{\gamma-1}+ | \partial_x \sqrt{\rho}|^2 \bigg] dxdt +\int_{0}^{T} \phi \int_{\mathbb{T}} (\mu(\rho) | \partial_x u_{\epsilon} |^2 dxdt  \\ & \le \phi(0) \int_{\mathbb{T}} \bigg[\dfrac{1}{2}\rho_0 |u_0|^2+ \dfrac{\rho_0^\gamma}{\gamma-1}+ | \partial_x \sqrt{\rho_0}|^2 \bigg] dx+\dfrac{1}{2} \int_{0}^{T} \phi \int_{\mathbb{T}} \sum_{k=1}^{\infty} \rho |F_{k}|^2 dxdt \\ & + \sum_{k=1}^{\infty} \int_{0}^{T} \phi \int_{\mathbb{T}} \rho F_{k} u dxdW_k , 
\end{split}
\end{equation}
holds for all $ \phi \in C_c^\infty ([0,T)),  \; \phi \ge 0, \; \mathbb{P}-a.s.$
\\
\item[(10)] 
The entropy inequality
\begin{equation}\label{entropy weak def}
\begin{split}
&  -\int_{0}^{T} \partial_t \phi  \int_{\mathbb{T}} \bigg[\dfrac{1}{2}\rho |V|^2+ \dfrac{\rho^\gamma}{\gamma-1}+ | \partial_x \sqrt{\rho}|^2 \bigg] dxdt + \dfrac{4 \gamma}{(\gamma+\alpha -1)^2} \int_{0}^{T} \phi \int_{\mathbb{T}} | \partial_x \rho^{\frac{\gamma+\alpha-1}{2} } |^2 dxdt \\ & + \dfrac{4}{\alpha^2} \int_{0}^{T} \phi \int_{\mathbb{T}} | \partial_{xx} 	\rho^{\frac{\alpha}{2}} |^2 dxdt +\dfrac{4(4-3 \alpha) }{3 \alpha^3}\int_{0}^{T} \phi \int_{\mathbb{T}} \rho^{-\alpha} | \partial_x \rho^\frac{\alpha}{2}|^4 dxdt  \\ & \\ & \le  \phi(0) \int_{\mathbb{T}} \bigg[\dfrac{1}{2}\rho_0 |V_0|^2+ \dfrac{\rho_0^\gamma}{\gamma-1}+ | \partial_x \sqrt{\rho_0}|^2 \bigg] dx+\dfrac{1}{2} \int_{0}^{T} \phi \int_{\mathbb{T}} \sum_{k=1}^{\infty} \rho |F_{k}|^2 dxdt \\ & + \sum_{k=1}^{\infty} \int_{0}^{T} \phi \int_{\mathbb{T}} \rho F_{k} V dxdW_k , 
\end{split}
\end{equation}
holds for all $ \phi \in C_c^\infty ([0,T)),  \; \phi \ge 0, \; \mathbb{P}-a.s.$
\end{itemize}
\end{definition}
\begin{remark}
Conditions \eqref{energy weak def} and \eqref{entropy weak def} are motivated from the fact that solutions of \eqref{stoc quantum}-\eqref{C.I} formally satisfies the following energy balance:
\begin{equation}\label{energy def strong}
\begin{split}
& \int_{\mathbb{T}} \dfrac{1}{2} \rho |u|^2+ \dfrac{\rho^\gamma}{\gamma-1}+ | \partial_x \sqrt{\rho}|^2 dx + \int_{0}^{t} \int_{\mathbb{T}} \mu(\rho)| \partial_x u |^2 dxds\\ & =\int_{\mathbb{T}} \dfrac{1}{2} \rho_0 |u_0|^2+ \dfrac{\rho_0^\gamma}{\gamma-1}+ | \partial_x \sqrt{\rho_0}|^2 dx +  \frac{1}{2} \sum_{k=1}^{\infty}\int_{0}^{t}\int_{\mathbb{T}} \rho |F_k|^2dxds + \int_{0}^{t} \int_{\mathbb{T}} \rho u \mathbb{F}(\rho,u) dxdW
\end{split}
\end{equation}
and entropy balance 
\begin{equation}\label{entropy def strong}
\begin{split}
& \int_{\mathbb{T}} \dfrac{1}{2}\rho |V|^2+ \dfrac{\rho^\gamma}{\gamma-1}+ | \partial_x \sqrt{\rho}|^2 dx + \dfrac{4 \gamma}{(\gamma+\alpha -1)^2} \int_{0}^{t}\int_{\mathbb{T}} | \partial_x \rho^{\frac{\gamma+\alpha-1}{2} } |^2 dxds \\ & + \dfrac{4}{\alpha^2} \int_{0}^{t} \int_{\mathbb{T}} | \partial_{xx} 	\rho^{\frac{\alpha}{2}} |^2 dxds +\dfrac{4(4-3 \alpha) }{3 \alpha^3}\int_{0}^{t} \int_{\mathbb{T}} \rho^{-\alpha} | \partial_x \rho^\frac{\alpha}{2}|^4 dxds \\ & = \int_{\mathbb{T}} \dfrac{1}{2}\rho_0 |V_0|^2+ \dfrac{\rho_0^\gamma}{\gamma-1}+ | \partial_x \sqrt{\rho_0}|^2 dx +\frac{1}{2} \sum_{k=1}^{\infty}\int_{0}^{t}\int_{\mathbb{T}} \rho |F_k|^2dxds+ \int_{0}^{t}\int_{\mathbb{T}} \rho V \mathbb{F}(\rho,u) dxdW,
\end{split}
\end{equation}
where $V$ is the effective velocity  $$V= u+ Q$$
and $$Q= \mu(\rho) \dfrac{\partial_x \rho}{\rho^2}$$
satisfies the following identity
$$ \partial_t Q+ u \partial_x Q= - \rho^{-1}\partial_x(\mu (\rho) \partial_{x} u ).$$
The balance laws \eqref{energy def strong}-\eqref{entropy def strong} are formally obtained by applying It$\hat{\text{o}}$ lemma, Theorem \ref{Ito Lemma} with $s=\rho, \; r=u, \;  Q(r)=\frac{1}{2}|u|^2$  to the functional 
\begin{equation}
F(\rho,u)(t) =\int_{\mathbb{T}} \frac{1}{2}\rho(t)|u(t)|^2dx
\end{equation}
and 
 $s=\rho, \; r=V, \;  Q(r)=\frac{1}{2}|V|^2$  to the functional 
\begin{equation}
F(\rho,u)(t) =\int_{\mathbb{T}} \frac{1}{2}\rho(t)|V(t)|^2dx,
\end{equation}
respectively.
\end{remark}
The main theorem of our paper is the following 
\begin{theorem}\label{main theorem weak}
Let $(\rho_0, u_0)$ be $\mathfrak{F}_0$-measurable random variables satisfying \eqref{C.I STRONG}.  Let the coefficients $G_k$ satisfy the conditions \eqref{G1}-\eqref{G2}. Then there exists a weak dissipative martingale solution $((\Omega, \mathfrak{F}, (\mathfrak{F}_t)_{t \ge 0}, \mathbb{P}), \rho, \Lambda, \zeta)$ to system \eqref{stoc quantum}-\eqref{C.I} in the sense of Definition \ref{def weak dissip sol}.
\end{theorem}
\begin{remark}
For $\alpha \in (\frac{1}{2}, 1]$ the regularity inferred by the a priori estimates for system \eqref{stoc quantum} is not enough to avoid the arising of the vacuum states of the density $\rho.$ Consequently the velocity $u$ may be not uniquely defined on the set $ \{ \rho=0 \}.$
\end{remark}
\begin{remark}\label{Remark 2}
The regularity of the initial conditions $(\rho_0, u_0)$ and of the stochastic forcing term $\mathbb{G}$ can be weakened by applying a suitable regularization argument. Indeed the minimal hypothesis for which Theorem \ref{main theorem weak} holds are the following:
\begin{equation}
 C^{-1} \le \rho_0 \le C,  \quad \rho_0 \in L^p(\Omega; H^{1}(\mathbb{T})),  \quad u_0 \in L^p(\Omega; L^2(\mathbb{T})),  \quad \text{for all} \; p \in [1,\infty).
\end{equation}
and $$ G_k: \mathbb{T} \times [0, \infty) \times \mathbb{R} \rightarrow \mathbb{R}, \quad G_k \in C^1(\mathbb{T} \times [0, \infty) \times \mathbb{R})$$ 
satisfying
\begin{equation}\label{hp 1 on G}
G_k(x,\rho, \rho u)| \le g_k( \rho + \rho |u|),
\end{equation}
\begin{equation}\label{hp 2 on G}
|\partial_{\rho, \rho u }G_k(x,\rho, \rho u)| \le g_k,
\end{equation}
\begin{equation}\label{hp 3 on G}
\sum_{k=1}^{\infty} g_k^2 < \infty.
\end{equation}
By defining $\rho_{0,\epsilon}= \rho_0 \ast \eta_\epsilon, \; u_{0,\epsilon}= u_0 \ast \eta_\epsilon$ and $\mathbb{F}_\epsilon= \mathbb{F} \ast \eta_\epsilon$ with $\eta_\epsilon$ standard mollifier,  we have that the approximating initial data and stochastic forcing term satisfy the hypothesis of Theorem \ref{global well posedness epsilon}.  In this case  the convergence result of the initial data must be included in the stochastic compactness argument through the tightness of the laws $\mathcal{L}[\rho_{0,\epsilon}], \; \mathcal{L}[u_{0,\epsilon}].$
\end{remark}
\begin{remark}
The result holds even in the deterministic case $\mathbb{G}=0$ and for an additive noise $\sigma(x)dW$ i.e.  when the noise does not depend on the solution. 
\end{remark}
\section{The approximating system}
In order to prove the existence of weak dissipative martingale solutions to system \eqref{stoc quantum} we construct an approximating system which enjoys extra dissipation properties and for which we prove the global well-posedness in the framework of strong solutions. 
In particular,  for $(x,t) \in \mathbb{T} \times (0,T),$ the approximating system is as follows
\begin{equation} \label{eps stoc quantum}
\begin{cases}
\text{d}\rho_{\epsilon}+\partial_x(\rho_{\epsilon} u_{\epsilon})\text{d}t=0,\\
\text{d}(\rho_{\epsilon} u_{\epsilon})+[ \partial_x(\rho_{\epsilon} u_{\epsilon}^2+\rho_{\epsilon}^\gamma)]\text{d}t=[\partial_x(\rho_{\epsilon}^\alpha\partial_x u_{\epsilon})+\epsilon \partial_{xx}u_{\epsilon} +  \rho_{\epsilon} \partial_x \bigg( \dfrac{\partial_{xx} \sqrt{\rho_{\epsilon}}}{\sqrt{\rho_{\epsilon}}}\bigg) ]\text{d}t +\mathbb{G}(\rho_{\epsilon},\rho_{\epsilon} u_{\epsilon})\text{d}W,
\end{cases}
\end{equation}
endowed with the following initial conditions:
\begin{equation} \label{epsilon C.I}
(\rho_{\epsilon}(x,t), u_{\epsilon}(x,t))_{|_{t=0}}=(\rho_0(x), u_0(x))
\end{equation} 
and periodic boundary conditions 
\begin{equation}\label{epsilon bc}
\rho_{\epsilon}(0,t)=\rho_{\epsilon}(1,t), \quad u_{\epsilon}(0,t)=u_{\epsilon}(1,t).
\end{equation}
Our goal is to perform the limit $\epsilon \rightarrow 0$ in order to prove the convergence of the strong solution $(\rho_\epsilon, u_\epsilon)$ of the approximating system \eqref{eps stoc quantum}-\eqref{epsilon C.I} to a weak dissipative martingale solution of system \eqref{stoc quantum}-\eqref{C.I}. For the system \eqref{eps stoc quantum}-\eqref{epsilon bc} it is possible to prove a well-posedness results.
We start by introducing the definition of strong pathwise solutions for system \eqref{eps stoc quantum}.
\begin{definition}(Local strong pathwise solution). \label{Def1}
Let $(\Omega, \mathfrak{F},(\mathfrak{F}_t)_{t \ge 0},\mathbb{P}), $ be a stochastic basis with a complete right-continuous filtration and let $W$ be an $(\mathfrak{F}_t)$-cylindrical Wiener process. Let $(\rho_0,u_0)$ be an $\mathfrak{F}_0-$measurable random variable with regularity 
\begin{equation} \label{C.I STRONG loc}
 C^{-1} \le \rho_0 \le C,  \quad \rho_0 \in H^{s+1}(\mathbb{T}),  \quad u_0 \in H^s(\mathbb{T}) \quad \mathbb{P}\text{-a.s.}
\end{equation}
for some $s > \frac{5}{2}.$
A triplet $(\rho_\epsilon,u_\epsilon,\tau)$ is called a local strong pathwise solution to system \eqref{eps stoc quantum}-\eqref{epsilon C.I}, provided
\begin{itemize}
\item[(1)]
 $\tau$ is a $\mathbb{P}-a.s.$ strictly positive $(\mathfrak{F}_t)$-stopping time;
 \item[(2)]
 the density $\rho_\epsilon$ is a $H^{s+1}(\mathbb{T})$-valued $(\mathfrak{F}_t)$-progressively measurable stochastic process such that 
 $$ \rho_\epsilon(\cdot \land \tau) >0, \quad \rho_\epsilon(\cdot \land \tau) \in C([0,T];H^{s+1}(\mathbb{T})) \quad \mathbb{P}-a.s.; $$
 \item[(3)]
 the velocity $u_\epsilon$ is a $H^s(\mathbb{T})$-valued $(\mathfrak{F}_t)$-progressively measurable stochastic process such that 
 $$ u_\epsilon(\cdot  \land \tau) >0, \quad u_\epsilon(\cdot \land \tau) \in C([0,T];H^s(\mathbb{T})) \quad \mathbb{P}-a.s.; $$
\item[(4)]
the continuity equation $$ \rho_\epsilon(t \land \tau)= \rho_0- \int_{0}^{t\land \tau} \partial_x(\rho_\epsilon u_\epsilon) ds $$
holds for all $t \in [0,T] \ \mathbb{P}-a,s,;$
\item[(5)] the momentum equation 
\begin{equation*}
\begin{split}
(\rho_\epsilon u_\epsilon)(t \land \tau)&= \rho_0 u_0 -\int_{0}^{t \land \tau} \partial_x(\rho_\epsilon u_\epsilon^2+p(\rho_\epsilon))ds +\int_{0}^{t \land \tau} \partial_x((\mu(\rho_\epsilon)+\epsilon) \partial_x u_\epsilon)ds \\ &+\int_{0}^{t \land \tau} \rho_\epsilon \partial_x \bigg( \dfrac{\partial_{xx}\sqrt{\rho_\epsilon}}{\sqrt{\rho_\epsilon}}\bigg)ds + \int_{0}^{t \land \tau} \mathbb{G}(\rho_\epsilon,\rho_\epsilon u_\epsilon)dW
\end{split}
\end{equation*}
holds for all $t \in [0,T] \ \mathbb{P}-a,s,;$
\end{itemize}
\end{definition}
\noindent
We underline that the presence of the third order derivative term in the momentum equation provides a different linear structure of the equation with respect to the compressible Navier-Stokes equations which determines a different regularity class for strong solutions $$(\rho_\epsilon,u_\epsilon) \in H^{s+1}(\mathbb{T}) \times H^s(\mathbb{T}).$$ For the proof of the local existence result of the system \eqref{eps stoc quantum}-\eqref{epsilon bc},  we refer to Theorem 2.3 in \cite{D.P.S.}.
\\
Now we introduce the definition of maximal and global strong pathwise solution
\begin{definition}(Maximal strong pathwise solution). \label{Def2}
Let $(\Omega, \mathfrak{F},(\mathfrak{F}_t)_{t \ge 0},\mathbb{P}), $ be a stochastic basis with a complete right-continuous filtration and let $W$ be an $(\mathfrak{F}_t)$-cylindrical Wiener process. Let $(\rho_0,u_0)$ be an $\mathfrak{F}_0-$measurable random variable in the space $H^{s+1}(\mathbb{T})\times H^s(\mathbb{T}).$
A quadruplet $(\rho_\epsilon,u_\epsilon,(\tau_R)_{R \in \mathbb{N}},\tau)$ is called a maximal strong pathwise solution to system \eqref{eps stoc quantum}-\eqref{epsilon C.I}, provided
\begin{itemize}
\item[(1)]
 $\tau$ is a $\mathbb{P}-a.s.$ strictly positive $(\mathfrak{F}_t)$-stopping time;
 \item[(2)] 
 $(\tau_R)_{R \in \mathbb{N}}$ is an increasing sequence of $(\mathfrak{F}_t)$-stopping times such that $\tau_R < \tau$ on the set $[\tau < \infty], \ \lim_{R \rightarrow \infty} \tau_R=\tau$ a.s. and 
\begin{equation}\label{limit cond}
\max \bigg\{ \sup_{t \in [0,\tau_R]} \| \log\rho_\epsilon (t) \|_{W_x^{2,\infty}},\;\sup_{t \in [0,\tau_R]}  \| u_\epsilon(t) \|_{W_x^{2,\infty}} \bigg\} \ge R \ \text{in} \ [\tau < \infty];
\end{equation}
\item[(3)]
each triplet $(\rho_\epsilon,u_\epsilon,\tau_R),  R\in \mathbb{N},$ is a local strong pathwise solution in the sense of Definition \ref{Def1}
\end{itemize}
If $(\rho_\epsilon,u_\epsilon,(\tau_R)_{R \in \mathbb{N}},\tau)$ is a maximal strong pathwise solution  and $\tau=\infty$ a.s.  then we say that the solution is global. 
\end{definition}
For system \eqref{eps stoc quantum} we infer the following well-posedness result which has been obtained in \cite{D.P.S.},  Theorem 2.4 for a similar viscosity profile.
\begin{theorem}\label{global well posedness epsilon}
Let $s \in \mathbb{N},  s > \frac{7}{2}, \; \gamma> 1.$ Let the coefficients $G_k$ satisfy \eqref{G1}-\eqref{G2}.  Let the initial datum $(\rho_0(x), u_0(x))$ be $\mathfrak{F}_0$-measurable random variables satisfying \eqref{C.I STRONG}. Then there exists a unique maximal global strong pathwise solution to system \eqref{eps stoc quantum}-\eqref{epsilon C.I} in the sense of Definition \ref{Def2}. 
In addition, the following regularity estimates hold:
\begin{itemize}
\item[(1)] Energy estimate:
\begin{equation}\label{epsilon energy visk}
\begin{split}
& \mathbb{E} \bigg| \sup_{t \in [0,T ]} \int_{\mathbb{T}} \bigg[\dfrac{1}{2}\rho_{\epsilon} |u_{\epsilon}|^2+ \dfrac{\rho_{\epsilon}^\gamma}{\gamma-1}+ | \partial_x \sqrt{\rho_{\epsilon}}|^2 \bigg] dx \bigg|^p +\mathbb{E} \bigg| \int_{0}^{T} \int_{\mathbb{T}} (\mu(\rho_{\epsilon}) +\epsilon)| \partial_x u_{\epsilon} |^2 dxdt \bigg|^p \\ & \lesssim 1+ \mathbb{E}\bigg|  \int_{\mathbb{T}} \bigg[\dfrac{1}{2}\rho_0|u_0|^2+ \dfrac{\rho_0^\gamma}{\gamma-1}+ | \partial_x \sqrt{\rho_0}|^2 \bigg] dx \bigg|^p,
\end{split}
\end{equation}
for any $p\ge 1.$
\\
\item[(2)] BD Entropy estimate: 
\begin{equation}\label{epsilon BD ENTROPY visk}
\begin{split}
& \mathbb{E} \bigg| \sup_{t \in [0,T]} \int_{\mathbb{T}} \bigg[\dfrac{1}{2}\rho_{\epsilon} |V_{\epsilon}|^2+ \dfrac{\rho_{\epsilon}^\gamma}{\gamma-1}+ | \partial_x \sqrt{\rho_{\epsilon}}|^2 \bigg] dx \bigg|^p \\ & +\mathbb{E} \bigg|  \dfrac{4 \gamma}{(\gamma+\alpha -1)^2} \int_{0}^{T}\int_{\mathbb{T}} | \partial_x \rho_{\epsilon}^{\frac{\gamma+\alpha-1}{2} } |^2 dxdt + \dfrac{4\epsilon \gamma}{(\gamma-1)^2} \int_{0}^{T}\int_{\mathbb{T}} | \partial_x \rho_{\epsilon}^{\frac{\gamma-1}{2} } |^2 dxdt \\ & + \dfrac{4}{\alpha^2} \int_{0}^{T} \int_{\mathbb{T}} | \partial_{xx} \rho_{\epsilon}^{\frac{\alpha}{2}} |^2 dxdt +\dfrac{4(4-3 \alpha) }{3 \alpha^3}\int_{0}^{T} \int_{\mathbb{T}} \rho_{\epsilon}^{-\alpha} | \partial_x \rho_{\epsilon}^\frac{\alpha}{2}|^4 dxdt \\ & +\dfrac{\epsilon}{2} \int_{0}^{T} \int_{\mathbb{T}} | \partial_{xx} \log \rho_{\epsilon} |^2 dxdt  \bigg|^p \\ & \lesssim 1+ \mathbb{E}\bigg|  \int_{\mathbb{T}} \bigg[\dfrac{1}{2}\rho_0 |V_0|^2+ \dfrac{\rho_0^\gamma}{\gamma-1}+ | \partial_x \sqrt{\rho_0}|^2 \bigg] dx \bigg|^p 
\end{split}
\end{equation}
for any $p\ge 1,$
where $V_{\epsilon}$ is the effective velocity  $$V_{\epsilon}= u_{\epsilon}+ Q_{\epsilon}^{\alpha}+\epsilon Q_{\epsilon}^0$$
and $$Q_{\epsilon}^{\alpha}= \mu(\rho_{\epsilon}) \dfrac{\partial_x \rho_{\epsilon}}{\rho_{\epsilon}^2},  \quad Q_{\epsilon}^0=  \dfrac{\partial_x \rho_{\epsilon}}{\rho_{\epsilon}^2}$$ verify the following transport equations
$$ \text{d} Q_{\epsilon}^{\alpha}+ u_{\epsilon} \partial_x Q_{\epsilon}^{\alpha}\text{d}t= - \rho^{-1}\partial_x(\mu (\rho_\epsilon) \partial_{x} u_\epsilon )\text{d}t$$
$$ \epsilon (\text{d} Q_{\epsilon}^0+ u_{\epsilon} \partial_x Q_{\epsilon}^0\text{d}t)= - \rho_{\epsilon}^{-1}(\epsilon \partial_{xx} u_{\epsilon} )\text{d}t.$$
\\
\item[(3)] No vacuum states:
\begin{equation}\label{upper rho}
\sqrt{\rho_{\epsilon}} \in L^p(\Omega;L^{\infty}(0,T;L^\infty(\mathbb{T}))),
\end{equation}
and
\begin{equation}\label{lower rho}
\dfrac{1}{\sqrt{\rho_{\epsilon}}} \in L^p(\Omega;L^{\infty}(0,T;L^\infty(\mathbb{T}))),
\end{equation}
for all $ p\in [1, \infty).$
\end{itemize}
\end{theorem}
\begin{remark}
The estimates \eqref{epsilon energy visk}-\eqref{upper rho} are independent on $\epsilon$ while the lower bound \eqref{lower rho} for the density degenerates as $\epsilon \rightarrow 0.$
\end{remark}
\begin{proof}
Here we present a sketch of the proof of Theorem \ref{global well posedness epsilon}.  We refer the reader to \cite{D.P.S.} for the detailed computations.  Precisely the existence of a unique local maximal strong patwhise solutions is obtained by improving the technique developed in \cite{Feir},  Chapter 5 for compressible fluids.  The extension argument relies on the derivation of the a priori estimates which allows to control the vacuum states of the density and high order derivative estimates. 
In order to prove \eqref{epsilon energy visk} we apply It$\hat{\text{o}}$ Lemma (Theorem \ref{Ito Lemma}) for $s=\rho_\epsilon, \; r=u_\epsilon, \;  Q(r)=\frac{1}{2}|u_\epsilon|^2$  to the functional 
\begin{equation}
F(\rho_\epsilon,u_\epsilon)(t) =\int_{\mathbb{T}} \frac{1}{2}\rho_\epsilon(t)|u_\epsilon(t)|^2dx,
\end{equation}
and after integrating by parts we deduce the following balance law
\begin{equation}
\begin{split}
& \int_{\mathbb{T}} \dfrac{1}{2} \rho_\epsilon |u_\epsilon|^2+ \dfrac{\rho_\epsilon^\gamma}{\gamma-1}+ | \partial_x \sqrt{\rho_\epsilon}|^2 dx + \int_{0}^{t} \int_{\mathbb{T}} (\mu(\rho_\epsilon)+\epsilon)| \partial_x u_\epsilon |^2 dxds\\ & =\int_{\mathbb{T}} \frac{1}{2} \rho_0 |u_0|^2 + \dfrac{\rho_0^\gamma}{\gamma-1}+ | \partial_x \sqrt{\rho_0}|^2 dx+\frac{1}{2} \sum_{k=1}^{\infty}\int_{0}^{t}\int_{\mathbb{T}} \rho_\epsilon |F_k|^2dxds+\int_{0}^{t} \int_{\mathbb{T}} \rho_\epsilon u_\epsilon \mathbb{F}(\rho_\epsilon,u_\epsilon) dxdW.
\end{split}
\end{equation}
Then by taking the $\sup$ in time,  the $p$-th power and applying the expectation we have
\begin{equation}\label{energy stoch}
\begin{split}
& \mathbb{E} \bigg| \sup_{t \in [0.T]} \int_{\mathbb{T}} \bigg[ \dfrac{1}{2} \rho_\epsilon |u_\epsilon|^2+ \dfrac{\rho_\epsilon^\gamma}{\gamma-1}+ | \partial_x \sqrt{\rho_\epsilon}|^2 \bigg] dx \bigg|^p + \mathbb{E} \bigg| \int_{0}^{T} \int_{\mathbb{T}} (\mu(\rho_\epsilon)+\epsilon)| \partial_x u_\epsilon |^2 dxdt \bigg|^p \\ &  \lesssim \mathbb{E} \bigg|  \int_{\mathbb{T}} \bigg[ \dfrac{1}{2} \rho_0 |u_0|^2+ \dfrac{\rho_0^\gamma}{\gamma-1}+ | \partial_x \sqrt{\rho_0}|^2 \bigg] dx \bigg|^p+ \mathbb{E} \bigg | \frac{1}{2} \sum_{k=1}^{\infty}\int_{0}^{T}\int_{\mathbb{T}} \rho_\epsilon |F_k|^2dxdt \bigg|^p \\ & +\mathbb{E} \bigg| \int_{0}^{T} \int_{\mathbb{T}} \rho_\epsilon u_\epsilon \mathbb{F}(\rho_\epsilon,u_\epsilon) dxdW \bigg|^p.
\end{split}
\end{equation}
The stochastic integral in the right hand side of \eqref{energy stoch} is estimated by means of the Burkholder-Davis-Gundy inequality,  Proposition \ref{BDG}, as follows 
\begin{equation}\label{stoch int enrgy estimate}
\begin{split}
& \mathbb{E} \bigg| \int_{0}^{t} \int_{} \rho_\epsilon u_\epsilon \mathbb{F}(\rho_\epsilon,u_\epsilon) dxdW \bigg|^p \le  \mathbb{E} \bigg[ \sup_{t \in [0,T]} \bigg|  \int_{0}^{t} \int_{} \rho_\epsilon u_\epsilon \mathbb{F}(\rho_\epsilon,u_\epsilon) dxdW \bigg|^p \bigg] \\ & \lesssim \mathbb{E} \bigg[ \int_{0}^{T} \sum_{k \in \mathbb{N}} \bigg| \int_{\mathbb{T}} \rho_\epsilon u_\epsilon F_k dx \bigg|^2 dt \bigg]^{\frac{p}{2}}\le \mathbb{E}\bigg[ \int_{0}^{T}  \bigg| \int_{\mathbb{T}} ( \rho_\epsilon+ \rho_\epsilon |u_\epsilon|^2) dx  \bigg|^2 dt \bigg]^{\frac{p}{2}},
\end{split}
\end{equation}
while the second contribution due to the stochastic forcing term in the right hand side of \eqref{energy stoch} is controlled by 
\begin{equation}
\begin{split}
\sum_{k \in \mathbb{N}} \int_{0}^{T} \int_{\mathbb{T}} \rho_\epsilon |F_k|^2dxdt \le \int_{0}^{T} \sum_{k \in \mathbb{N}} \alpha^2_k \int_{\mathbb{T}} ( \rho_\epsilon + \rho_\epsilon |u_\epsilon|^2) dxdt \lesssim \int_{0}^{T} \int_{\mathbb{T}} ( \rho_\epsilon+ \rho_\epsilon |u_\epsilon|^2) dxdt,
\end{split}
\end{equation}
Finally \eqref{epsilon energy visk} is obtained by applying Young inequality in \eqref{stoch int enrgy estimate} and using the Gronwall Lemma.  \\ The derivation of \eqref{epsilon BD ENTROPY visk} is obtained by using the same lines of argument of \eqref{epsilon energy visk} by replacing the velocity $u$ with the effective velocity $V_\epsilon=u_{\epsilon}+ Q_{\epsilon}^{\alpha}+\epsilon Q_{\epsilon}^0$ in the It$\hat{\text{o}}$ Lemma (Theorem \ref{Ito Lemma}). \\
\\
Equation \eqref{upper rho} follows by the energy estimate \eqref{epsilon energy visk} and Sobolev embedding.  
Concerning \eqref{lower rho} we observe that by mean value theorem,  for any $t \in (0,T)$ and $\omega \in \Omega,$ there exists a point $\bar x$ such that $$\rho_{\epsilon}(\bar x,t,  \omega)= \int_{\mathbb{T}} \rho_{\epsilon}(y,t,\omega) dy.$$
Therefore,  by the conservation of the mass and \eqref{C.I STRONG} there exists a deterministic constant $C>0 $ such that 
\begin{equation}\label{eq:a}
\frac{1}{C} \le \rho_{\epsilon}(\bar x,t,\omega) \le C, \quad \mathbb{P}\text{-a.s.}
\end{equation}
Now we define $$H_{\epsilon}=c(\alpha) \rho_{\epsilon}^{\alpha-\frac{1}{2}}+ \epsilon c(0) \rho_{\epsilon}^{-\frac{1}{2}},$$ where $c(\alpha)= \frac{1}{(\alpha-\frac{1}{2})},  \; c(0)=-2.$
By using the fundamental theorem of calculus, for any fixed $t \in (0,T), \; \omega \in \Omega$ and $x \in \mathbb{T},$ the following identity holds
\begin{equation}\label{FTC }
 H_{\epsilon}(x, t, \omega)-H_{\epsilon}(\bar{x},t, \omega)= \int_{\bar{x}}^{x} \partial_x H_{\epsilon}(y,t,\omega)dy.
\end{equation}
Hence we have 
\begin{equation}\label{FTC ineq}
|H_{\epsilon}(x, t, \omega)| \le |H_{\epsilon}(\bar{x},t, \omega)| +  \int_{\bar{x}}^{x} |\partial_x H_{\epsilon}(y,t,\omega)|dy,
\end{equation}
by taking the $\sup$ in both space and time and the $p-$th power in \eqref{FTC ineq} after applying expectation we get
\begin{equation}
H_{\epsilon} \in L^p(\Omega;L^{\infty}(0,T;L^\infty(\mathbb{T}))),
\; \text{for all} \; p\in [1, \infty).
\end{equation}
Moreover we observe that 
\begin{equation}
\begin{split}
| \epsilon c(0) \rho_{\epsilon}^{-\frac{1}{2}}| \le | c(\alpha) \rho_{\epsilon}^{\alpha-\frac{1}{2}}| +| c(\alpha) \rho_{\epsilon}^{\alpha-\frac{1}{2}}+ \epsilon c(0) \rho_{\epsilon}^{-\frac{1}{2}}|= | c(\alpha) \rho_{\epsilon}^{\alpha-\frac{1}{2}}| + | H_{\epsilon}|,
\end{split}
\end{equation}
hence by taking the sup in both space and time,  the p-th power and the expectation we get 
\begin{equation}
\begin{split}
\epsilon^p |c(0)|^p \mathbb{E} \bigg| \sup_{t,x} \rho_{\epsilon}^{-\frac{1}{2}} \bigg|^p\le c^p(\alpha)\mathbb{E} \bigg| \sup_{t,x} \rho_{\epsilon}^{\alpha-\frac{1}{2}} \bigg|^p+\mathbb{E} \bigg| \sup_{t,x} | H_{\epsilon}| \bigg|^p \le C(\alpha),
\end{split}
\end{equation}
from which we deduce 
\begin{equation}
\epsilon^p |c(0)|^p  \| \frac{1}{\sqrt{\rho_{\epsilon}}} \|_{L^p_\omega L^\infty_t L^\infty_x} \le C(\alpha).
\end{equation}
\end{proof}
\section{Proof of the main result}
This Section is devoted to the proof of the main result.  In particular, we adapt the strategy used in \cite{Lacroix} and \cite{Spirito} by introducing a suitable truncation of the velocity field $u$ in the momentum equation of \eqref{eps stoc quantum}. This method, together with a stochastic compactness argument,  allows us to prove the convergence of the strong pathwise solution $(\rho_\epsilon, u_\epsilon)$ determined in Section 4 to a limit process $(\rho, \Lambda, \zeta)$ which we identify as a weak dissipative martingale solutions of system \eqref{stoc quantum}-\eqref{C.I} in the sense of Definition \ref{def weak dissip sol}.
\subsection{The truncated formulation}
Let $\bar{\beta}: \mathbb{R} \rightarrow \mathbb{R}$ be an even,  positive,  smooth function with compact support such that \begin{equation*}
\bar{\beta}(z)=1, \quad z \in [-1,1],
\end{equation*}
\begin{equation*}
\text{supp} \bar{\beta} \subset (-2,2),
\end{equation*}
\begin{equation*}
0 \le \bar{\beta} \le 1.
\end{equation*}
Let also $\tilde{\beta}:\mathbb{R} \rightarrow \mathbb{R}$ be the antiderivative of $\bar{\beta}$ i.e. 
$$ \tilde{\beta}(z)= \int_{0}^{z} \bar{\beta}(s)ds$$
and $\bar{\beta}_\delta(z)=\bar{\beta}(\delta z), \;\tilde{\beta}_\delta(z)=\tilde{\beta}(\delta z).$ We define
$$\beta_\delta(z)= \dfrac{1}{\delta}\tilde{\beta}_\delta(z)$$
and we observe that $\beta_\delta(z)$ is an approximation of the identity.
The next lemma collects some elementary properties of the above truncation.  The proof can be performed by direct computations.
\begin{lemma}\label{lemma beta}
Let $\delta >0$ and $K= \| \bar{\beta} \|_{W^{2,\infty}}.$ Then for any $\delta >0$ there exists $C=C(K)$ such that the following hold
\begin{equation}\label{beta hp}
\beta_\delta(t) \underset{\delta \rightarrow 0}{ \longrightarrow t},  \quad \beta^{'}_\delta(t) \underset{\delta \rightarrow 0}{ \longrightarrow 1},  \quad| \beta^{'}_\delta(t) | \le 1, \quad | \beta^{''}_\delta (t) | \le C\delta.
\end{equation}
\end{lemma}
\noindent
Now we infer that the truncated velocity $\beta(u_\epsilon)$ satisfies the following formulation of the momentum equation
\begin{proposition}\label{truncated form prop}
Let $(\rho_\epsilon, u_\epsilon)$ be a strong pathwise solution of system \eqref{eps stoc quantum}-\eqref{epsilon C.I},  let $\beta$ be the truncation defined in \eqref{beta hp},  then the following identity holds
\begin{equation}\label{truncated momentum}
\begin{split}
& \int_{\mathbb{T}}\rho_{\epsilon} \beta_\delta(u_{\epsilon})dx= \int_{\mathbb{T}}\rho_0 \beta_\delta(u_0)dx-\int_{0}^{t}\int_{\mathbb{T}}[2 \rho_\epsilon^{\frac{\gamma}{2}}\partial_x \rho_\epsilon^{\frac{\gamma}{2}} \beta^{'}_\delta(u_\epsilon)]dxds \\ & +\int_{0}^{t}\int_{\mathbb{T}} [N_\epsilon \beta^{'}_\delta (u_\epsilon)+\partial_x(\rho_{\epsilon}^{\frac{\alpha}{2}}M_\epsilon \beta^{'}_\delta (u_\epsilon))- \rho_\epsilon^{\frac{\alpha}{2}} M_\epsilon \beta^{''}_\delta (u_\epsilon) \partial_x u_\epsilon]dxds \\ & +\int_{0}^{t} \int_{\mathbb{T}} \sum_{k \in \mathbb{N}} \dfrac{1}{2} \rho_\epsilon \beta^{''}_\delta(u_\epsilon) |F_k|^2 dxds + \int_{0}^{t} \int_{\mathbb{T}} \sum_{k\in \mathbb{N}} \rho_\epsilon \beta^{'}_\delta (u_\epsilon) F_k dx\text{d}W,
\end{split}
\end{equation}
with $$M_\epsilon= \zeta_\epsilon + \rho_\epsilon^{\frac{1-\alpha}{2}}(\partial_{xx} \sqrt{\rho_\epsilon}-4| \partial_x \rho_\epsilon^{\frac{1}{4}}|^2), \quad N_\epsilon= \epsilon \partial_{xx}u_\epsilon.$$
\end{proposition}
\begin{proof}
We apply It$\text{\'{o}}$ lemma,  Theorem \ref{Ito Lemma} for $s= \rho_\epsilon, \; r=u_\epsilon, \; Q(r)=\beta_\delta(u_\epsilon)$ and we deduce 
\begin{equation}\label{ito balance}
\begin{split}
& \int_{\mathbb{T}}\rho_{\epsilon} \beta_\delta(u_{\epsilon})dx= \int_{\mathbb{T}}\rho_0 \beta_\delta(u_0)dx-\int_{0}^{t}\int_{\mathbb{T}}[\rho_\epsilon \beta^{'}_\delta(u_\epsilon) u_\epsilon \partial_x u_{\epsilon}-2 \rho_\epsilon^{\frac{\gamma}{2}}\partial_x \rho_\epsilon^{\frac{\gamma}{2}} \beta^{'}_\delta(u_\epsilon)]dxds \\ & +\int_{0}^{t}\int_{\mathbb{T}} [N_\epsilon \beta^{'}_\delta (u_\epsilon)+\beta^{'}_\delta(u_\epsilon) \partial_x(\rho_{\epsilon}^{\frac{\alpha}{2}}M_\epsilon)]dxds +\int_{0}^{t} \int_{\mathbb{T}} \sum_{k \in \mathbb{N}} \dfrac{1}{2} \rho_\epsilon \beta^{''}_\delta(u_\epsilon) |F_k|^2 dxds+ \\ & -\int_{0}^{t}\int_{\mathbb{T}} \beta_\delta(u_\epsilon) \partial_x (\rho_\epsilon u_\epsilon) dxds + \int_{0}^{t} \sum_{k\in \mathbb{N}} \rho_\epsilon \beta^{'}_\delta (u_\epsilon) F_k dx\text{d}W.
\end{split}
\end{equation}
Now we observe that since $\rho_\epsilon \beta^{'}_\delta (u_\epsilon) u_\epsilon \partial_x u_\epsilon = \rho_\epsilon u_\epsilon \partial_x \beta_\delta (u_\epsilon),$ 
then the second term in the right hand side of \eqref{ito balance} cancels with the last non noise term after integrating by parts.
Furthermore 
$\beta^{'}_\delta \partial_x(\rho_{\epsilon}^{\frac{\alpha}{2}}M_\epsilon)= \partial_x(\rho_{\epsilon}^{\frac{\alpha}{2}}M_\epsilon \beta^{'}_\delta (u_\epsilon))- \rho_\epsilon^{\frac{\alpha}{2}} M_\epsilon \beta^{''}_\delta (u_\epsilon) \partial_x u_\epsilon$
and hence we end up with \eqref{truncated momentum}.
\end{proof}
\subsection{Uniform bounds}
For the reader simplicity we collect the $\epsilon-$independent bounds from \eqref{epsilon energy visk}, \eqref{epsilon BD ENTROPY visk} and \eqref{upper rho}.  To be precise we have that there exists a deterministic constant $C>0,$ independent on $\epsilon$ such that the following regularity hold
\begin{equation}\label{eps indep reg}
\begin{split}
& \| \sqrt{\rho_\epsilon}u_\epsilon \|_{L^p_\omega L^\infty_t L^2_x}\le C,  \quad \| \rho_\epsilon\|_{L^p_\omega L^\infty_t L^\gamma_x} \le C, \quad  \| \partial_x \sqrt{\rho_\epsilon}\|_{L^p_\omega L^\infty_t L^2_x} \le C, \\ & \| \rho_\epsilon^{\frac{\alpha}{2}}\partial_x u_\epsilon \|_{L^p_\omega L^2_t L^2_x} \le C,  \quad  \| \partial_x (\rho_\epsilon^{\alpha-\frac{1}{2}} ) \|_{L^p_\omega L^{\infty}_t L^2_x} \le C, \quad \| \partial_x (\rho_\epsilon^{\frac{\gamma+\alpha-1}{2} }) \|_{ L^p_\omega L^{2}_t L^2_x} \le C,  \quad \\ & \| \partial_{xx} \rho^{\frac{\alpha}{2}} \|_{L^p_\omega L^{2}_t L^2_x}, \le C,   \quad  \| \partial_x \rho_\epsilon^\frac{\alpha}{4} \|_{L^p_\omega L^{4}_t L^4_x} \le C, \quad \| \sqrt{\rho_\epsilon} \|_{L^p_\omega L^\infty_t L^\infty_x} \le C,  
\end{split}
\end{equation}
for all $ p \in [1,\infty).$ Moreover from \eqref{eps indep reg} we deduce \begin{equation} \label{eps indep reg 2}
\begin{split}
\| \rho_\epsilon u_\epsilon \|_{L^p_\omega L^2_t L^2_x} \le C,  \quad \| \partial_x(\rho_\epsilon u_\epsilon) \|_{L^p_\omega L^2_t L^1_x} \le C,
\end{split}
\end{equation}
and by using the continuity equation we have
\begin{equation}
\| \partial_t \rho_\epsilon \| _{L^p_\omega L^2_t L^1_x} \le C.
\end{equation}
Finally, we observe that 
\begin{equation}
\begin{split}
& \| \sqrt{\epsilon} \partial_x u_\epsilon \|_{L^p_\omega L^2_t L^2_x} \le C, \quad  \| \sqrt{\epsilon} \partial_x (\rho^{\frac{\gamma-1}{2} }) \|_{ L^p_\omega L^{2}_t L^2_x} \le C, \\ & \| \sqrt{\epsilon} \partial_{xx} \log \rho_\epsilon \|_{ L^p_\omega L^{2}_t L^2_x} \le C, \quad \| \epsilon \frac{1}{\sqrt{\rho_{\epsilon}}} \|_{L^p_\omega L^\infty_t L^\infty_x} \le C.
\end{split}
\end{equation}
\subsection{Stochastic compactness}
Stochastic compactness relies on the tightness of the family of joint law of the process $(\rho_\epsilon,  u_\epsilon).$ In particular,  due to the regularity of strong solutions we make use of the Jakubowski-Skorokhod representation theorem, Theorem \ref{Jak} to prove (up to a subsequence) the convergence of the solution in a suitable path space $\mathcal{X}.$
\\
\\
Using the above uniform bounds \eqref{eps indep reg}-\eqref{eps indep reg 2} our goal is to perform the limit $\epsilon \rightarrow 0.$
Let $\nu$ be a probability measure on $L^1(\mathbb{T}) \times L^1(\mathbb{T})$
satisfying 
\begin{equation}
\nu \{ \rho \ge0\}=1, \quad  \int_{L^1_x \times L^1_x} \bigg| \int_{\mathbb{T}} \bigg[\dfrac{1}{2}\rho |u|^2+ \dfrac{\rho^\gamma}{\gamma-1}+ | \partial_x \sqrt{\rho}|^2 dx \bigg|^r d\nu(\rho,\rho u ) < \infty, \quad r\ge 1.
\end{equation}
and let $(\Omega, \mathfrak{F},(\mathfrak{F}_t)_{t \ge 0},\mathbb{P}), $ be a stochastic basis with a complete right-continuous filtration and let $W$ be an $(\mathfrak{F}_t)$-cylindrical Wiener process.  We define $\nu_\epsilon= \mathbb{P} \circ (\rho_0,u_0)^{-1}$ and by definition of $\nu_\epsilon$ we have
\begin{equation}\label{reg mu}
\int_{L^1_x \times L^1_x} \bigg| \int_{\mathbb{T}} \bigg[\dfrac{1}{2}\rho_\epsilon |u_\epsilon|^2+ \dfrac{\rho_\epsilon^\gamma}{\gamma-1}+ | \partial_x \sqrt{\rho_\epsilon}|^2 dx \bigg|^r d\nu_\epsilon(\rho,\rho u ) < \infty, \quad r\ge 1.
\end{equation}
and 
\begin{equation}
\begin{split}
& \int_{L^1_x \times L^1_x} \bigg| \int_{\mathbb{T}} \bigg[\dfrac{1}{2}\rho_\epsilon |u_\epsilon|^2+ \dfrac{\rho_\epsilon^\gamma}{\gamma-1}+ | \partial_x \sqrt{\rho_\epsilon}|^2 dx \bigg|^r d\nu_\epsilon(\rho,\rho u ) \\ & \longrightarrow \int_{L^1_x \times L^1_x} \bigg| \int_{\mathbb{T}} \bigg[\dfrac{1}{2}\rho |u|^2+ \dfrac{\rho^\gamma}{\gamma-1}+ | \partial_x \sqrt{\rho}|^2 dx \bigg|^r d\nu(\rho,\rho u ).
\end{split}
\end{equation}
Let $(\rho_\epsilon,u_\epsilon)$ be a strong solution of \eqref{eps stoc quantum}-\eqref{epsilon C.I}, in accordance to the regularity estimates \eqref{eps indep reg} we define the following path space $$\mathcal{X}=\mathcal{X}_\rho\times\mathcal{X}_m\times\mathcal{X}_\zeta \times\mathcal{X}_\Lambda \times \mathcal{X}_W \times \mathcal{X}_{\mu^{d,\delta}} \times \mathcal{X}_{\mu^{s,\delta}},$$
where $$\mathcal{X}_\rho= L^2(0,T;H^1(\mathbb{T})),\quad \mathcal{X}_m= L^p(0,T;L^p(\mathbb{T})), \quad \mathcal{X}_\zeta= (L^2(0,T;L^2(\mathbb{T})), w),$$
$$\mathcal{X}_\Lambda= (L^2(0,T;L^2(\mathbb{T})), w),\quad \mathcal{X}_W= C([0,T];\mathfrak{U}_0),$$ $$\mathcal{X}_{\mu^{d,\delta}}=(\mathcal{M}_b((0,T) \times \mathbb{T})), w^{\star}), \quad \mathcal{X}_{\mu^{s,\delta}}=(\mathcal{M}_b((0,T) \times \mathbb{T})), w^{\star}),$$
\\
and $$\mu_\epsilon^{d,\delta}= \rho_\epsilon^{\frac{\alpha}{2}} M_\epsilon \beta^{''}_\delta (u_\epsilon) \partial_x u_\epsilon  \quad \mu^{s,\delta}_\epsilon= \sum_{k \in \mathbb{N}} \dfrac{1}{2} \rho_\epsilon \beta_\delta^{''}(u_\epsilon) | F_k|^2.$$
Now we prove that the set $ \{ \mathcal{L}[ \rho_\epsilon, m_\epsilon, \zeta_\epsilon,\Lambda_\epsilon, W,{\mu_\epsilon^{d,\delta}},{\mu_\epsilon^{s,\delta}}]; \epsilon \in (0,1) \}$ is tight on $\mathcal{X}.$
In the following Proposition we establish the tightness of $\mathcal{L}[\rho_\epsilon]$ on $\mathcal{X}_\rho.$
\begin{proposition}
The set $\{\mathcal{L}[\rho_\epsilon]; \epsilon \in (0,1) \}$ is tight on $\mathcal{X}_\rho.$
\end{proposition}
\begin{proof}
We divide the proof in several steps.  First, we observe that from \eqref{upper rho} and \eqref{epsilon energy visk} we have $$\sqrt{\rho_\epsilon} \in L^p(\Omega;L^{\infty}(0,T;L^\infty(\mathbb{T}))), \quad \partial_x \sqrt{\rho_\epsilon} \in L^p(\Omega;L^{\infty}(0,T;L^2(\mathbb{T}))), \; \text{for all} \; p\in [1, \infty).$$
Hence we conclude that $\rho_\epsilon$ is uniformly bounded in $L^p(\Omega;L^{\infty}(0,T;H^1(\mathbb{T}))).$ Moreover from the continuity equation we have 
\begin{equation*}
\partial_t \rho_\epsilon+ u_\epsilon \partial_x \rho_\epsilon + \rho_\epsilon \partial_x u_\epsilon =0,
\end{equation*}
from which we deduce that $$\mathbb{E} \| \partial_t \rho_\epsilon \|^p_{ L^\infty(0,T; L^1(\mathbb{T}))}\le C.$$
Therefore we can make use of the following compact embedding 
$$L^2(0,T;H^1(\mathbb{T}) \cap H^1(0,T; L^1(\mathbb{T})) \overset{C}{\hookrightarrow} L^2(0,T;L^2(\mathbb{T})),$$
which follows directly from Aubin-Lion's Lemma,  to deduce that $ \{\mathcal{L}[\rho_\epsilon]; \epsilon \in (0,1) \} $ is tight on $L^2(0,T;L^2(\mathbb{T})).$
More precisely,  for all $L > 0,$ the set $$B_L= \{ \rho \in L^2(0,T;H^1(\mathbb{T}) \cap H^1(0,T; L^1(\mathbb{T})) \; \big| \;  \| \rho \|_{ L^2(0,T;H^1(\mathbb{T}) \cap H^1(0,T; L^1(\mathbb{T}))} \le L \}$$
is relatively compact in $L^2(0,T;L^2(\mathbb{T}))$ and we have 
\begin{equation}\label{abst tightness}
\begin{split}
\mathcal{L}[\rho_\epsilon](B_L^c) & =\mathbb{P}( \| \rho_\epsilon \|_{ L^2(0,T;H^1(\mathbb{T}) \cap H^1(0,T; L^1(\mathbb{T}))} \ge L) \\ & \le \frac{1}{L} \mathbb{E} \| \rho_\epsilon \|_{ L^2(0,T;H^1(\mathbb{T}) \cap H^1(0,T; L^1(\mathbb{T}))} \le \frac{C}{L}
\end{split}
\end{equation}
and by choosing $L$ sufficiently large we get the result. \\ \\
As a next step we want to prove the tightness of $\{\mathcal{L}[\rho_\epsilon]; \epsilon \in (0,1) \}$ on $L^2(0,T;H^1(\mathbb{T})).$
We start by recalling that from \eqref{eps indep reg} we have that $\partial_{xx}\rho_\epsilon^{\frac{\alpha}{2}}$ is uniformly bounded in $L^p(\Omega;L^{2}(0,T;L^2(\mathbb{T}))).$ 
This allow us to prove strong convergence of $\partial_x \rho_\epsilon^\frac{\alpha}{2}, $ indeed
\begin{equation}\label{strong rho H^1}
\begin{split}
& \int_{0}^{T} \int_{\mathbb{T}} | \partial_x \rho_\epsilon^\frac{\alpha}{2}-\partial_x \rho^\frac{\alpha}{2}|^2dxdt=\int_{0}^{T} \int_{\mathbb{T}} | \partial_x(\rho_\epsilon^\frac{\alpha}{2}-\rho^\frac{\alpha}{2})|^2dxdt \\ & \le \bigg( \int_{0}^{T} \int_{\mathbb{T}} | \rho_\epsilon^\frac{\alpha}{2}-\rho^\frac{\alpha}{2}|^2dxdt \bigg)^{\frac{1}{2}} \bigg( \int_{0}^{T} \int_{\mathbb{T}} | \partial_{xx} \rho_\epsilon^\frac{\alpha}{2}|^2+ |\partial_{xx} \rho^\frac{\alpha}{2}|^2 dxdt \bigg)^{\frac{1}{2}},
\end{split}
\end{equation}
hence by taking the p-th power,  the expectation and using the uniform boundness of the second derivative terms in \eqref{strong rho H^1},  we send $\epsilon \rightarrow 0$ and we get
\begin{equation}
\rho_\epsilon^\frac{\alpha}{2}\rightarrow \rho^\frac{\alpha}{2} \in L^p(\Omega; L^2(0,T; H^1(\mathbb{T}))).
\end{equation}
Therefore the tightness of  $\{\mathcal{L}[\rho_\epsilon]; \epsilon \in (0,1) \}$ on $\mathcal{X}_\rho$ follows form \eqref{abst tightness} and an interpolation argument.  Indeed the set $K=B_L \cap \{ \rho_\epsilon \in L^2(0,T;H^1(\mathbb{T}) \: \big| \; \| \partial_x \rho_\epsilon \|_{L^2(0,T;L^2(\mathbb{T})} \le \tilde{L} \}$ is compact in $L^2(0,T;H^1(\mathbb{T}))$ (it is a closed subset of a compact set) and thus 
\begin{equation}\label{Tight markov}
\mathcal{L}[\rho_\epsilon](K^c)\le \frac{C}{L}+ \mathbb{P}( \| \partial_x \rho_\epsilon \|_{L^2_{t,x}} \ge \tilde{L}) \le \frac{C}{L}+ \frac{1}{\tilde{L}} \mathbb{E} \| \partial_x \rho_\epsilon \|_{L^2_{t,x}} \le \frac{C}{L}+\frac{\tilde{C}}{\tilde{L}}
\end{equation}
and by choosing $L, \tilde{L}$ sufficiently large we get the result. 
\end{proof}
\begin{proposition}
The set $\{\mathcal{L}[m_\epsilon]; \epsilon \in (0,1) \}$ is tight on $\mathcal{X}_m.$
\end{proposition}
\begin{proof}
The tightness of the marginals $\mathcal{L}[m_\epsilon]$ on $\mathcal{X}_m$ follows by a compact embedding and an interpolation argument.  More precisely we have that
$$ \mathbb{E} \| \partial_x (\rho_\epsilon u_\epsilon) \|^p_{L^2(0,T;L^1(\mathbb{T})} \le C$$
and 
$$ \mathbb{E} \| \rho_\epsilon u_\epsilon \|^p_{L^\infty(0,T;L^2(\mathbb{T})} \le C.$$
Moreover from the momentum equation \begin{equation*}
\partial_t(\rho_\epsilon u_\epsilon)+ \partial_x(\rho_\epsilon u^2_\epsilon)+ \partial_x \rho^\gamma_\epsilon-N_\epsilon= \partial_x(\rho_\epsilon^{\frac{\alpha}{2}}M_\epsilon)+ \mathbb{G}(\rho_\epsilon, \rho_\epsilon u_\epsilon)
\end{equation*}
with $$M_\epsilon= \zeta_\epsilon + \rho_\epsilon^{\frac{1-\alpha}{2}}(\partial_{xx} \sqrt{\rho_\epsilon}-4| \partial_x \rho_\epsilon^{\frac{1}{4}}|^2), \quad N_\epsilon= \epsilon \partial_{xx}u_\epsilon,$$
we observe that $$\mathbb{E} \| \partial_t (\rho_\epsilon u_\epsilon) \|^p_{L^2(0,T;W^{-1,1}(\mathbb{T}))} \le C$$
and therefore by using the following compact embedding 
$$L^2(0,T;W^{1,1}(\mathbb{T})) \cap H^1(0,T; W^{-1,1}(\mathbb{T})) \overset{C}{\hookrightarrow} L^2(0,T;L^1(\mathbb{T})),$$
which is a direct consequence of Aubin-Lion's Lemma,  we have that for any $L >0$ $$B_L= \{ m \in L^2(0,T;W^{1,1}(\mathbb{T})) \cap H^1(0,T; W^{-1,1}(\mathbb{T})) \; \big| \| m \|_{L^2(0,T;W^{1,1}(\mathbb{T})) \cap H^1(0,T; W^{-1,1}(\mathbb{T}))} \le L \}$$
is relatively compact in $L^2(0,T;L^1(\mathbb{T})).$ \\
\\
Moreover 
\begin{equation*}
\begin{split}
& \mathcal{L}[\rho_\epsilon u_\epsilon](B_L^c)=\mathbb{P}( \| \rho_\epsilon u_\epsilon \|_{L^2(0,T;W^{1,1}(\mathbb{T})) \cap H^1(0,T; W^{-1,1}(\mathbb{T}))} \ge L) \\ & \le \frac{1}{L} \mathbb{E} \| \rho_\epsilon u_\epsilon \|_{L^2(0,T;W^{1,1}(\mathbb{T})) \cap H^1(0,T; W^{-1,1}(\mathbb{T}))} \le \frac{C}{L}
\end{split}
\end{equation*}
hence the tightness of $\{\mathcal{L}[m_\epsilon]; \epsilon \in (0,1) \}$ on $L^2(0,T;L^1(\mathbb{T}))$ follows by choosing $L$ sufficiently large.
In order to improve the integrability we observe that since $$\rho_\epsilon u_\epsilon \in L^p(\Omega;L^\infty(0,T;L^2(\mathbb{T})) \cap L^p(\Omega;L^2(0,T;L^\infty(\mathbb{T}))$$
uniformly in $\epsilon,$ then by using the following interpolation inequality 
$$ \| \rho_\epsilon u_\epsilon \|_{L^r_{t,x}} \le \| \rho_\epsilon u_\epsilon \|^{\frac{r-2}{r}}_{L^\infty_t L^2_x}\| \rho_\epsilon u_\epsilon \|^{\frac{2}{r}}_{L^2_t L^\infty_x}
\quad \text{for any}\; r \in [2,\infty),$$
we get that $\rho_\epsilon u_\epsilon$ is uniformly bounded in $L^p(\Omega; L^r(0,T;L^r(\mathbb{T}))$ for any $r\in [1,\infty)$ and hence the desired tightness follows by rewriting $$B_L= \{ m \in Y= L^2(0,T;W^{1,1}(\mathbb{T})) \cap H^1(0,T; W^{-1,1}(\mathbb{T})) \cap L^r(0,T;L^r(\mathbb{T}) \; \big| \| m \|_Y \le L \}$$
and observing that it is relatively compact in $L^p(0,T;L^p(\mathbb{T}))$ for any $p\in[1,\infty)$ and applying a similar argument to the one used to deduce \eqref{Tight markov}.
\end{proof}
\begin{proposition}\label{tight lambda}
The set $\{\mathcal{L}[\zeta_\epsilon]; \epsilon \in (0,1) \}$ is tight on $\mathcal{X}_\zeta.$
\end{proposition}
\begin{proof}
For any $L>0,$ the set $$B_L= \{ \zeta \in L^2(0,T;L^2(\mathbb{T}); \| \zeta \|_{L^2_{t,x}} \le L \}$$
is relatively compact in $\mathcal{X}_\zeta$ by Banach-Alaoglu-Bourbaki Theorem. Moreover by using Markov inequality and energy estimate we have 
\begin{equation*}
\mathcal{L}[\zeta_\epsilon](B_L^c)=\mathbb{P}( \| \zeta_\epsilon \|_{L^2_{t,x}} \ge L) \le \frac{1}{L} \mathbb{E} \| \zeta_\epsilon \|_{L^2_{t,x}} \le \frac{C}{L}
\end{equation*}
Hence the tightness of $\{\mathcal{L}[\zeta_\epsilon]; \epsilon \in (0,1) \}$ on $\mathcal{X}_\zeta.$ follows by choosing $L$ sufficiently large.
\end{proof}
\begin{proposition}\label{tight zeta}
The set $\{\mathcal{L}[\Lambda_\epsilon]; \epsilon \in (0,1) \}$ is tight on $\mathcal{X}_\Lambda.$
\end{proposition}
\begin{proof}
Similarly to the proof of Proposition \ref{tight lambda},  since $\Lambda \in L^p(\Omega;L^\infty(0,T;L^2(\mathbb{T}))),$ then by considering $$B_L= \{ \Lambda \in L^2(0,T;L^2(\mathbb{T}); \| \Lambda \|_{L^2_{t,x}} \le L \}$$ we have that $B_L$ is relatively compact on $\mathcal{X}_\Lambda$ and
\begin{equation*}
\mathcal{L}[\Lambda_\epsilon](B_L^c)=\mathbb{P}( \| \Lambda_\epsilon \|_{L^2_{t,x}} \ge L) \le \frac{1}{L} \mathbb{E} \| \Lambda_\epsilon \|_{L^2_{t,x}} \le \frac{C}{L}.
\end{equation*}
\end{proof}
\begin{proposition}\label{tight W}
The set $\{\mathcal{L}[W] \}$ is tight on $\mathcal{X}_W.$
\end{proposition}
\begin{proof}
This claim follows easily by observing that $\mathcal{L}[W]$ is a Radon measure on a Polish space $C([0,T]); \mathfrak{U}_0).$
\end{proof}
\begin{proposition}\label{tight meas1}
The set $\{\mathcal{L}[\mu^{d,\delta}], \; \delta \in (0,1) \}$ is tight on $\mathcal{X}_{\mu^{d,\delta}}.$
\end{proposition}
\begin{proof}
We observe that the total variation 
\begin{equation}
\sup_\epsilon \big | \mu_\epsilon^{d,\delta} \big |\le \sup_\epsilon \int_{0}^{T} \int_{\mathbb{T}}  \rho_\epsilon^{\frac{\alpha}{2}} | M_\epsilon| \beta^{''}_\delta (u_\epsilon)| \partial_x u_\epsilon | dxdt \le C.
\end{equation}
is uniformly bounded by a deterministic constant $C=C(\delta)>0.$ 
Since by means of Riesz Theorem $ C((0,T) \times \mathbb{T})^{\star}= \mathcal{M}_b ((0,T) \times \mathbb{T}),$ the set $B_L= \{ \mu^{d,\delta} \in \mathcal{M}_b((0,T) \times \mathbb{T}); \big | \mu^{d,\delta} \big | \le L \}$ is compact with respect to the weak-star topology on $\mathcal{M}_b((0,T) \times \mathbb{T}).$ Moreover we observe that
\begin{equation}
\mathbb{E} \bigg| \int_{0}^{T} \int_{\mathbb{T}} \rho_\epsilon^{\frac{\alpha}{2}} | M_\epsilon| \beta^{''}_\delta (u_\epsilon)| \partial_x u_\epsilon | dxdt \bigg| \le C\delta  \| M_\epsilon \|_{ L^2_{\omega,t,x}} \| \rho^{\frac{\alpha}{2}}_\epsilon \partial_x u_\epsilon \|_{ L^2_{\omega,t,x}} \le C \delta
\end{equation}
and we deduce 
\begin{equation}\label{law meas1}
\mathcal{L}[\mu^{d,\delta}_\epsilon](B_L^c)=\mathbb{P}( \big | \mu^{d,\delta}_\epsilon \big | \ge L) \le \frac{1}{L} \mathbb{E} \big | \mu^{d,\delta}_\epsilon \big | \le \frac{C \delta}{L}.
\end{equation}
Hence by choosing $L$ sufficiently large we get the tightness.
\end{proof}
\begin{proposition}
The set $\{\mathcal{L}[\mu^{s,\delta}], \; \delta \in (0,1) \}$ is tight on $\mathcal{X}_{\mu^{d,\delta}}.$
\end{proposition}
\begin{proof}
Similarly to the proof of Proposition \ref{tight meas1} we have that 
\begin{equation}
\sup_\epsilon \big | \mu^{s,\delta}_\epsilon \big |\le \sup_\epsilon \int_{0}^{T} \int_{\mathbb{T}} \sum_{k \in \mathbb{N}} \dfrac{1}{2} \rho_\epsilon \beta_\delta^{''}(u_\epsilon) | F_k|^2 dxdt \le C \| \beta^{''}_\delta \|_{\infty} \sum_{k \in \mathbb{N}} \alpha^2_k\le C \delta,
\end{equation}
and also 
\begin{equation}
\mathbb{E} \bigg | \int_{0}^{T} \int_{\mathbb{T}} \sum_{k \in \mathbb{N}} \dfrac{1}{2} \rho_\epsilon \beta_\delta^{''}(u_\epsilon) | F_k|^2 dxdt \bigg| \le C\delta
\end{equation}
hence the tightness follows by using Riesz Theorem and the same lines of argument of Proposition \ref{tight meas1}.
\end{proof}
\begin{corollary}
The set $ \{ \mathcal{L}[ \rho_\epsilon, m_\epsilon, \zeta_\epsilon,\Lambda_\epsilon, W,{\mu_\epsilon^{d,\delta}},{\mu_\epsilon^{s,\delta}}]; \epsilon \in (0,1) \}$ is tight on $\mathcal{X}.$
\end{corollary}
\begin{proof}
Let $\theta \in (0,1),$ then by virtue of the tightness of the marginal laws we infer the existence of compact sets $K_\rho \subset \mathcal{X}_\rho, \; K_m \subset \mathcal{X}_m, \; K_\zeta \subset \mathcal{X}_\zeta, \; K_\Lambda \subset \mathcal{X}_\Lambda, \; K_W \subset \mathcal{X}_W, \; K_{\mu^{d,\delta}} \subset \mathcal{X}_{\mu^{d,\delta}}, \; K_{\mu^{s,\delta}}\subset \mathcal{X}_{\mu^{s,\delta}}$ such that $$\mathcal{L}[\rho_\epsilon](K_\rho)\ge 1-\frac{\theta}{7}, \quad \mathcal{L}[m_\epsilon](K_m)\ge 1-\frac{\theta}{7}, \quad \mathcal{L}[\zeta_\epsilon](K_\zeta)\ge 1-\frac{\theta}{7}, $$ $$ \mathcal{L}[\Lambda_\epsilon](K_\Lambda)\ge 1-\frac{\theta}{7}, \quad \mathcal{L}[W_\epsilon](K_W)\ge 1-\frac{\theta}{7}.$$
$$ \mathcal{L}[\mu^{d,\delta}_\epsilon](K_{\mu^{d,\delta}})\ge 1-\frac{\theta}{7}, \quad \mathcal{L}[\mu^{s,\delta}_\epsilon](K_{\mu^{s,\delta}})\ge 1-\frac{\theta}{7},$$
Hence by Tychonoff's Theorem $K_\rho \times K_m \times K_\zeta \times K_\Lambda \times K_W \times K_{\mu^{d,\delta}} \times K_{\mu^{s,\delta}}$ is compact in $\mathcal{X}$ and 
\begin{equation*}
\begin{split}
&\mathcal{L}[\rho_\epsilon, m_\epsilon, \zeta_\epsilon, \Lambda_\epsilon; W,{\mu_\epsilon^{d,\delta}},{\mu_\epsilon^{s,\delta}}](K_\rho \times K_m \times K_\zeta \times K_\Lambda \times K_W \times K_W \times K_{\mu^{d,\delta}} \times K_{\mu^{s,\delta}}) \\ & =\mathbb{P}(\rho_\epsilon \in K_\rho,  m_\epsilon \in K_m,  \zeta_\epsilon \in K_\zeta, \Lambda_\epsilon \in K_\Lambda,  W \in K_W,  \mu^{d,\delta}_\epsilon \in K_{\mu^{d,\delta}},\mu^{s,\delta}_\epsilon \in K_{\mu^{s,\delta}}) \\ & = 1- \mathbb{P}([ \rho_\epsilon \notin K_\rho] \cup [m_\epsilon \notin K_m] \cup [\zeta_\epsilon \notin K_\zeta] \cup [\Lambda_\epsilon \notin K_\Lambda] \cup [W \notin K_W]  \\ & \cup [\mu^{d,\delta}_\epsilon \notin K_{\mu^{d,\delta}}]\cup [\mu^{s,\delta}_\epsilon \notin K_{\mu^{s,\delta}}]) \\ & \ge 1- \mathbb{P}( \rho_\epsilon \notin K_\rho)-\mathbb{P} (m_\epsilon \notin K_m) -\mathbb{P}(\zeta_\epsilon \notin K_\zeta)-\mathbb{P}( \Lambda_\epsilon \notin K_\Lambda)-\mathbb{P}(W \notin K_W) - \\ & \mathbb{P}(\mu^{d,\delta}_\epsilon \notin K_{\mu^{d,\delta}} )-\mathbb{P} (\mu^{s,\delta}_\epsilon \notin K_{\mu^{s,\delta}})\\ & \ge 1-\theta.
\end{split}
\end{equation*}
\end{proof}
\noindent
As discussed in \cite{Feir}, since weak topologies are in general not metrizzable then the path space $\mathcal{X}$ is not a Polish space.  However it is a sub-Polish space due to the continuous embedding of $\mathcal{X}$ in $\mathcal{D}$'.  Hence we can make use of the Jakuboiwski-Skorokhod representation Theorem,  Theorem \ref{Jak} in order to infer the following convergence result
\begin{proposition}\label{conv stoch jak}
There exists a complete probability space  $(\tilde{\Omega}, \tilde{\mathfrak{F}},\tilde{\mathbb{P})}, $ with $\mathcal{X}$- valued Borel measurable random variables $[\tilde{\rho}_\epsilon,\tilde{m}_\epsilon, \tilde{\zeta}_\epsilon,\tilde{\Lambda}_\epsilon,\tilde{W}_\epsilon,{\mu_\epsilon^{d,\delta}},{\mu_\epsilon^{s,\delta}}]$ and $[\tilde{\rho},\tilde{m}, \tilde{\zeta},\tilde{\Lambda},\tilde{W},{\mu^{d,\delta}},{\mu^{s,\delta}}]$ such that (up to a subsequence):
\begin{itemize}
\item[(1)] For all $\epsilon \in (0,1), $ the laws of $[\tilde{\rho}_\epsilon,\tilde{m}_\epsilon, \tilde{\zeta}_\epsilon,\tilde{\Lambda}_\epsilon,\tilde{W}_\epsilon,{\mu_\epsilon^{d,\delta}},{\mu_\epsilon^{s,\delta}}]$ and $[\tilde{\rho},\tilde{m}, \tilde{\zeta},\tilde{\Lambda},\tilde{W},{\mu^{d,\delta}},{\mu^{s,\delta}}]$ coincide on $\mathcal{X}$
\item[(2)] The law of $[\tilde{\rho},\tilde{m}, \tilde{\zeta},\tilde{\Lambda},\tilde{W},{\mu^{d,\delta}},{\mu^{s,\delta}}]$ on $\mathcal{X}$ is a Radon measure
\item[(3)] $[\tilde{\rho}_\epsilon,\tilde{m}_\epsilon, \tilde{\zeta}_\epsilon,\tilde{\Lambda}_\epsilon,\tilde{W}_\epsilon,{\mu_\epsilon^{d,\delta}},{\mu_\epsilon^{s,\delta}}]$ converges in the topology of $\mathcal{X}, \; \mathbb{P}-a.s.$ to $[\tilde{\rho},\tilde{m}, \tilde{\zeta},\tilde{\Lambda},\tilde{W},{\mu^{d,\delta}},{\mu^{s,\delta}}],$ i.e.,
$$\tilde{\rho}_\epsilon \rightarrow \tilde{\rho} \in L^2(0,T;H^1(\mathbb{T})),$$
$$\tilde{m}_\epsilon \rightarrow \tilde{m} \in L^p(0,T;L^p(\mathbb{T})),$$
$$\tilde{\zeta}_\epsilon \rightharpoonup \tilde{\zeta} \in L^2(0,T;L^2(\mathbb{T})),$$
$$\tilde{\Lambda}_\epsilon \rightharpoonup \tilde{\Lambda} \in L^2(0,T;L^2(\mathbb{T})),$$
$$\tilde{W}_\epsilon \rightarrow \tilde{W} \in C(0,T;\mathfrak{U}_0),$$
$$\tilde{\mu}_\epsilon^{d,\delta} \overset{\star}{\rightharpoonup} \tilde{\mu}^{d,\delta} \in \mathcal{M}_b((0,T) \times \mathbb{T}),$$
$$\tilde{\mu}_\epsilon^{s,\delta} \overset{\star}{\rightharpoonup} \tilde{\mu}^{d,\delta} \in \mathcal{M}_b((0,T) \times \mathbb{T}),$$
$\tilde{\mathbb{P}}-a.s$ for all $p \in [1,\infty).$
\\
\\
Moreover $\tilde{\Lambda}$ is such that $\sqrt{\tilde{\rho}} \tilde{\Lambda}= \tilde{m}$ and the following additional convergence results hold
\item[(4)] $\partial_{xx} \tilde{\rho}_\epsilon^{\frac{\alpha}{2}} \rightharpoonup \partial_{xx} \tilde{\rho}^{\frac{\alpha}{2}} \in L^2(0,T;L^2(\mathbb{T})),\quad \tilde{\mathbb{P}}-a.s$ for all $p \in [1,\infty),$
\item[(5)] $\tilde{\rho}_\epsilon \rightarrow \tilde{\rho} \in L^p(0,T;L^p(\mathbb{T})),\quad \tilde{\mathbb{P}}-a.s$ for all $p \in [1,\infty),$
\item[(6)] $\partial_x \tilde{\rho}_\epsilon^{\frac{\alpha}{2}} \rightarrow \partial_x \tilde{\rho}^{\frac{\alpha}{2}} \in L^2(0,T;L^2(\mathbb{T})),\quad \tilde{\mathbb{P}}-a.s$ for all $p \in [1,\infty).$
\end{itemize}
\end{proposition}
\noindent
Concerning the stochastic setting,  we observe that since the limit $\tilde{\zeta}$ and $\tilde{\Lambda}$ are not stochastic processes in the classical sense,  it is convenient to work with random distributions.  We recall that for random distributions which are adapted and satisfy a suitable integrability assumption there exists a progressively measurable stochastic process belonging to the same class of equivalence,  see Lemma 2.2.18 in \cite{Feir} for a detailed discussion.  Hence we infer that the minimal regularity assumption on the integrands for which the stochastic integral is well-defined is the non anticipativity of the corresponding joint canonical filtration with respect to the Wiener process.  To this purpose we define 
\begin{equation}
\tilde{\mathfrak{F}}_t:= \sigma \bigg( \sigma_t [\tilde{\rho}] \cup \sigma_t [\tilde{\zeta}]\cup \sigma_t [\tilde{\Lambda}] \cup \bigcup_{k=1}^{\infty} \sigma_t [\tilde{W}_k] \bigg), \quad t \in [0,T]
\end{equation}
and we claim that $\tilde{\mathfrak{F}}_t$ is independent of $\sigma(\tilde{W}(s)-\tilde{W}(t))$ for all $s>t.$ \\ The process $\tilde{W}_\epsilon= \sum_{k=1}^{\infty} e_k \tilde{W}_{\epsilon,k}$ is a cylindrical Wiener process with respect to 
\begin{equation}
\sigma \bigg( \sigma_t [\tilde{\rho}_\epsilon] \cup \sigma_t [\tilde{\zeta}_\epsilon]\cup \sigma_t [\tilde{\Lambda}_\epsilon]\cup \bigcup_{k=1}^{\infty} \sigma_t [\tilde{W}_{\epsilon,k}] \bigg), \quad t \in [0,T]
\end{equation}
and this filtration is non-anticipative with respect to $\tilde{W}_\epsilon.$ This allow us to pass to the limit as $\epsilon \rightarrow 0$ in order to get the non-anticipativity of $(\tilde{\mathfrak{F}}_t)_{t \ge 0}$ with respect to $\tilde{W.}$ In order to conclude our claim,  we observe that by using Lemma 2.1.35 and Corollary 2.1.36 in \cite{Feir},  $\tilde{W}$ is a cylindrical Wiener process with respect to $(\tilde{\mathfrak{F}}_t)_{t \ge 0}.$
\\
\\
We conclude this part by proving the following Lemma 
\begin{lemma}\label{def u}
Let $f \in C \cap L^\infty(\mathbb{R}), \; p \in [1,\infty).$ Let $(\tilde{\rho}_\epsilon, \tilde{u}_\epsilon)$ be a strong pathwise solutions of $\eqref{eps stoc quantum}-\eqref{epsilon C.I}.$ Define 
$$\tilde{u}(x,t,\omega)= \begin{cases}
\dfrac{\tilde{m}(x,t,\omega)}{\tilde{\rho}(x,t,\omega)}=\dfrac{\tilde{\Lambda}(x,t,\omega)}{\sqrt{\tilde{\rho}(x,t,\omega)}}, \quad (x,t,\omega) \in  \{ \tilde{\rho}>0 \} \\
0, \quad \quad \quad \quad \quad \quad \quad \quad \quad \, \; \; \;\quad (x,t,\omega) \in \{ \tilde{\rho}=0 \}.
 \end{cases}$$
Then 
\begin{itemize}
\item [1)] 
\begin{equation}\label{conv rho f}
\tilde{\rho}_\epsilon f(\tilde{u}_\epsilon) \rightarrow \tilde{\rho} f(\tilde{u}) \;  \text{in} \; L^p(0,T;L^p(\mathbb{T})),\quad \text{for all} \; p\in[1,\infty), \; \tilde{\mathbb{P}}-a.s..
\end{equation}
\item [2)]
\begin{equation}\label{conv rho u f}
\tilde{\rho}_\epsilon \tilde{u}_\epsilon f(\tilde{u}_\epsilon) \rightarrow \tilde{\rho} \tilde{u} f(\tilde{u}) \;  \text{in} \; L^p(0,T;L^p(\mathbb{T})),\quad \text{for all} \; p\in[1,\infty), \; \tilde{\mathbb{P}}-a.s..
\end{equation}
\item [(3)]
\begin{equation}\label{conv der rho f}
\partial_x \tilde{\rho}_\epsilon^{\frac{\alpha}{2}} f(\tilde{u}_\epsilon) \rightarrow \partial_x \tilde{\rho}^{\frac{\alpha}{2}} f(\tilde{u})\;  \text{in} \; L^p(0,T;L^p(\mathbb{T})),\quad \text{for all} \; p\in[1,\infty), \; \tilde{\mathbb{P}}-a.s..
\end{equation}
\end{itemize}
\end{lemma}
\begin{proof}
We observe that by virtue of Proposition \ref{conv stoch jak} the a priori estimates \eqref{eps indep reg}-\eqref{eps indep reg 2} and Vitali's Theorem we have
\begin{equation}\label{rho conv omega t x}
\tilde{\rho}_\epsilon \rightarrow \tilde{\rho} \quad \text{in} \;L^1(\tilde{\Omega}; L^2(0,T);H^1(\mathbb{T})), 
\end{equation}
\begin{equation}
\tilde{\rho}_\epsilon \tilde{u}_\epsilon \rightarrow \tilde{m} \quad \text{in} \; L^1(\tilde{\Omega};L^p(0,T);L^p(\mathbb{T})), 
\end{equation}
\begin{equation}\label{der rho conv omega t x}
\partial_x \tilde{\rho}^{\frac{\alpha}{2}}_\epsilon \rightarrow \partial_x \tilde{\rho}^{\frac{\alpha}{2}} \quad  \text{in} \; L^1(\tilde{\Omega};L^2(0,T);L^2(\mathbb{T})).
\end{equation}
Hence we deduce that up to a subsequence 
\begin{equation}\label{rho a.e.}
\tilde{\rho}_\epsilon \rightarrow \tilde{\rho} \quad a.e. \quad  \text{in} \; \mathbb{T} \times (0,T)\times\tilde{\Omega}, 
\end{equation}
\begin{equation}\label{m a.e.}
\tilde{\rho}_\epsilon \tilde{u}_\epsilon \rightarrow \tilde{m} \quad a.e. \quad  \text{in} \; \mathbb{T} \times (0,T)\times\tilde{\Omega}, 
\end{equation}
\begin{equation}\label{dx rho a.e}
\partial_x \tilde{\rho}^{\frac{\alpha}{2}}_\epsilon \rightarrow \partial_x \tilde{\rho}^{\frac{\alpha}{2}} \quad a.e. \quad  \text{in} \; \mathbb{T} \times (0,T)\times\tilde{\Omega}.
\end{equation}
Moreover, by the Fatou lemma we have 
\begin{equation}
\int_{\tilde{\Omega}}\int_{0}^{T} \int_{\mathbb{T}} \underset{\epsilon \rightarrow 0}{\liminf} \dfrac{\tilde{m}^2_\epsilon}{\tilde{\rho}_\epsilon} dxdtd\tilde{P} \le \underset{\epsilon \rightarrow 0}{\liminf}\int_{\tilde{\Omega}}\int_{0}^{T} \int_{\mathbb{T}} \dfrac{\tilde{m}^2_\epsilon}{\tilde{\rho}_\epsilon} dxdtd\tilde{P}< \infty
\end{equation}
which implies that $\tilde{m}=0$ on the set $\{\tilde{\rho}=0\}$ and also 
$$\sqrt{\tilde{\rho}}\tilde{u} \in L^2(\Omega;L^2(0,T;L^2(\mathbb{T}))), \quad \tilde{m}=\tilde{\rho}\tilde{u} =\sqrt{\tilde{\rho}} \tilde{\Lambda}.$$
In order to prove \eqref{conv rho f},  we observe that by virtue of \eqref{rho a.e.} we have $$\tilde{\rho}_\epsilon f(\tilde{u}_\epsilon) \rightarrow \tilde{\rho} f(\tilde{u}) \quad a.e.  \quad \text{in} \; \{\tilde{\rho} >0 \}.$$
On the other hand,  since $f \in L^\infty(\mathbb{R})$ we get $$| \tilde{\rho}_\epsilon f(\tilde{u}_\epsilon)| \le | \tilde{\rho}_\epsilon | \| f \|_{\infty} \rightarrow 0  \quad a.e.  \quad \text{in} \; \{ \tilde{\rho}=0 \}$$
and, since $\{ \tilde{\rho}_\epsilon \}_\epsilon \subset L^p(0,T;L^p(\mathbb{T})), \; \tilde{\mathbb{P}}-a.s.$ then by using Vitali's theorem we get our claim.
Concerning \eqref{conv rho u f} we observe that \eqref{m a.e.} implies
$$ \tilde{\rho}_\epsilon \tilde{u}_\epsilon f(\tilde{u}_\epsilon) \rightarrow \tilde{m}  f(\tilde{u}) \quad a.e. \quad \text{in} \; \{ \tilde{\rho}>0\},$$
$$|  \tilde{\rho}_\epsilon \tilde{u}_\epsilon f(\tilde{u}_\epsilon) | \le | \partial_x \tilde{\rho}_\epsilon \tilde{u}_\epsilon  | \| f \|_{\infty} \rightarrow 0 \quad a.e. \quad \text{in} \; \{ \tilde{\rho}=0\}.$$
Hence by using the uniform bounds \eqref{eps indep reg} and Vitali's theorem we get \eqref{conv rho u f}.
Finally in order to conclude the proof we observe that since $\tilde{\rho}$ is a Sobolev function,  then (see \cite{Evans G}) $$\partial_x \tilde{\rho}=0 \quad a.e. \quad \text{in} \; \{\tilde{\rho}=0\}.$$
Moreover $$ \partial_x \tilde{\rho}_\epsilon f(\tilde{u}_\epsilon) \rightarrow \partial_x \tilde{\rho} f(\tilde{u}) \quad a.e. \quad \text{in} \; \{ \tilde{\rho}>0\},$$
$$| \partial_x \tilde{\rho}_\epsilon f(\tilde{u}_\epsilon) | \le | \partial_x \tilde{\rho}_\epsilon | \| f \|_{\infty} \rightarrow 0 \quad a.e. \quad \text{in} \; \{ \tilde{\rho}=0\}.$$
hence $$\partial_x \tilde{\rho}_\epsilon f(\tilde{u}_\epsilon) \rightarrow \partial_x \tilde{\rho} f(\tilde{u}) \quad a.e. \quad \text{in} \; \mathbb{T} \times (0,T)\times \tilde{\Omega}.$$ Again by virtue of the uniform bounds \eqref{eps indep reg} and Vitali's theorem we get our result.
\end{proof}
\noindent
\subsection{Proof of Theorem \ref{main theorem weak}}
Our next goal is to identify the limit $(\tilde{\rho},\tilde{\zeta},\tilde{\Lambda})$ as a weak dissipative martingale solution to \eqref{stoc quantum} in the sense of Definition \ref{def weak dissip sol}. We divide the proof in two steps.
\\
\\
\textbf{Step 1: The limit $\epsilon \rightarrow 0$}
\\
We pass to the limit in the weak formulation of \eqref{truncated momentum}. 
To this purpose let $(\tilde{\rho}_\epsilon, \tilde{u}_\epsilon)$ be a strong pathwise solution of \eqref{eps stoc quantum}-\eqref{epsilon C.I}, then $(\tilde{\rho}_\epsilon, \tilde{u}_\epsilon) \rightarrow (\tilde{\rho},\tilde{m},\tilde{\zeta},\tilde{\Lambda})$ in the regularity class of Proposition \ref{conv stoch jak}. 
Moreover defining $\tilde{u}$ as in Lemma \ref{def u} we have
$$ \sqrt{\tilde{\rho}}\tilde{u} \in L^\infty(0,T; L^2(\mathbb{T})), \quad \tilde{\zeta} \in L^2(0,T; L^2(\mathbb{T})), \quad \tilde{m}=\sqrt{\tilde{\rho}}\tilde{\Lambda}= \tilde{\rho} \tilde{u}, \quad \tilde{\mathbb{P}}-a.s.$$
As a consequence of the equality in laws form Proposition \ref{conv stoch jak} and Theorem \ref{eq in law equation} we have that \eqref{eps stoc quantum} is satisfied on the new probability space by $(\tilde{\rho}_\epsilon,\tilde{u}_\epsilon)$ and hence by using the uniform bounds \eqref{eps indep reg} we can perform the limit $\epsilon \rightarrow 0$ to deduce that 
\begin{equation}
\int_{0}^{T} \int_{\mathbb{T}} \tilde{\rho}_\epsilon \tilde{u}_\epsilon \partial_x \psi dxdt \longrightarrow \int_{0}^{T} \int_{\mathbb{T}} \tilde{\rho}\tilde{u} \partial_x \psi dxdt,
\end{equation}
for any $\psi \in C^{\infty}_c([0,T) \times \mathbb{T}),\quad \tilde{\mathbb{P}}-a.s..$
Hence the continuity equation is satisfied by $(\tilde{\rho},\tilde{\zeta},\tilde{\Lambda})$ in the sense of distributions.  Concerning the momentum equation we recall that strong solutions satisfy the truncated formulation \eqref{truncated momentum} also in the weak formulation 
\begin{equation}\label{truncated weak epsilon}
\begin{split}
& \int_{\mathbb{T}} \tilde{\rho}_0 \beta_\delta(\tilde{u}_0)\psi(0,x)dx +\int_{0}^{T} \int_{\mathbb{T}} \tilde{\rho}_{\epsilon} \beta_\delta(\tilde{u}_{\epsilon}) \partial_t \psi dxdt+\int_{0}^{T} \int_{\mathbb{T}}\tilde{\rho}_{\epsilon} \tilde{u}_{\epsilon} \beta_\delta (\tilde{u}_\epsilon)\partial_x \psi dxdt \\ & + \int_{0}^{T} \int_{\mathbb{T}}2 \tilde{\rho}_\epsilon^{\frac{\gamma}{2}}\partial_x \tilde{\rho}_\epsilon^{\frac{\gamma}{2}} \beta^{'}_\delta(\tilde{u}_\epsilon) \psi dxdt- \int_{0}^{T} \int_{\mathbb{T}}\tilde{N}_\epsilon \beta^{'}_\delta (\tilde{u}_\epsilon)\psi dxdt -\int_{0}^{T} \int_{\mathbb{T}} \tilde{\rho}_{\epsilon}^{\frac{\alpha}{2}}\tilde{M}_\epsilon \beta^{'}_\delta (\tilde{u}_\epsilon) \partial_x \psi dxdt \\ & +\int_{0}^{T} \int_{\mathbb{T}}  \tilde{\rho}_\epsilon^{\frac{\alpha}{2}} \tilde{M}_\epsilon \beta^{''}_\delta (\tilde{u}_\epsilon) \partial_x \tilde{u}_\epsilon \psi dxdt + \int_{0}^{T} \int_{\mathbb{T}} \sum_{k \in \mathbb{N}} \dfrac{1}{2} \tilde{\rho}_\epsilon \beta_\delta^{''}(\tilde{u}_\epsilon) | F_k(\tilde{\rho}_\epsilon,\tilde{u}_\epsilon)|^2 \psi dxdt \\ & + \int_{0}^{T} \int_{\mathbb{T}} \sum_{k \in \mathbb{N}} \tilde{\rho}_\epsilon F_k(\tilde{\rho}_\epsilon,\tilde{u}_\epsilon)  \beta^{'}_\delta (\tilde{u}_\epsilon) \psi dx \text{d}\tilde{W}_k=0
\end{split}
\end{equation}
for any $\psi \in C^{\infty}_c([0,T) \times \mathbb{T}), \quad \tilde{\mathbb{P}}-a.s.$ \\
\\
We first perform the limit $\epsilon \rightarrow 0$ for $\delta$ fixed. 
First we observe that by using \eqref{conv rho f} and \eqref{beta hp} we have
\begin{equation}
\int_{0}^{T} \int_{\mathbb{T}} \tilde{\rho}_{\epsilon} \beta_\delta(\tilde{u}_{\epsilon}) \partial_t \psi  dxdt \longrightarrow \int_{0}^{T} \int_{\mathbb{T}} \tilde{\rho} \beta_\delta(\tilde{u}) \partial_t \psi dxdt,\quad \tilde{\mathbb{P}}-a.s.
\end{equation}
and, similarly, by \eqref{conv rho u f} we infer
\begin{equation}
\int_{0}^{T} \int_{\mathbb{T}} \tilde{\rho}_{\epsilon} \tilde{u}_{\epsilon} \beta_\delta (\tilde{u}_\epsilon)\partial_x \psi dxdt \longrightarrow \int_{0}^{T} \int_{\mathbb{T}} \tilde{\rho} \tilde{u} \beta_\delta (\tilde{u})\partial_x \psi dxdt,\quad \tilde{\mathbb{P}}-a.s..
\end{equation}
Concerning the pressure term, we observe that by using \eqref{conv der rho f} and \eqref{beta hp} we have 
\begin{equation}
\int_{0}^{T} \int_{\mathbb{T}} 2 \tilde{\rho}_\epsilon^{\frac{\gamma}{2}}\partial_x \tilde{\rho}_\epsilon^{\frac{\gamma}{2}} \beta^{'}_\delta(\tilde{u}_\epsilon) \psi dxdt \longrightarrow \int_{0}^{T} \int_{\mathbb{T}} 2 \tilde{\rho}^{\frac{\gamma}{2}}\partial_x \tilde{\rho}^{\frac{\gamma}{2}} \beta^{'}_\delta(\tilde{u}) \psi dxdt,\quad \tilde{\mathbb{P}}-a.s..
\end{equation}
Since by definition $$\tilde{M}_\epsilon= \tilde{\zeta}_\epsilon + \tilde{\rho}_\epsilon^{\frac{1-\alpha}{2}}(\partial_{xx} \sqrt{\tilde{\rho}_\epsilon}-4| \partial_x \tilde{\rho}_\epsilon^{\frac{1}{4}}|^2) \in L^2(0,T; L^2(\mathbb{T}),  \; \tilde{\mathbb{P}}-a.s.,$$ by using Proposition \ref{conv stoch jak}  we deduce that there exists $\tilde{M}$ such that $$\tilde{M}_\epsilon \rightharpoonup \tilde{M} \; \text{in} \; L^2(0,T; L^2(\mathbb{T}),  \; \tilde{\mathbb{P}}-a.s..$$ By virtue of \eqref{conv rho f} we infer
\begin{equation}
\int_{0}^{T} \int_{\mathbb{T}} \tilde{\rho}_{\epsilon}^{\frac{\alpha}{2}}\tilde{M}_\epsilon \beta^{'}_\delta (\tilde{u}_\epsilon) \partial_x \psi dxdt \longrightarrow \int_{0}^{T} \int_{\mathbb{T}} \tilde{\rho}^{\frac{\alpha}{2}}\tilde{M} \beta^{'}_\delta (\tilde{u}) \partial_x \psi dxdt,\quad \tilde{\mathbb{P}}-a.s..
\end{equation}
\\
For the extra dissipation term we observe that 
\begin{equation}
\epsilon \int_{0}^{T} \int_{\mathbb{T}} \partial_{xx} \tilde{u}_\epsilon \beta^{'}_\delta(\tilde{u}_\epsilon) \psi dxdt=-\epsilon \int_{0}^{T} \int_{\mathbb{T}} \partial_x \tilde{u}_\epsilon \beta^{''}_\delta (\tilde{u}_\epsilon)\partial_x \tilde{u}_\epsilon \psi -\epsilon \int_{0}^{T} \int_{\mathbb{T}} \partial_x \tilde{u}_\epsilon \beta^{'}_\delta (\tilde{u}_\epsilon) \partial_x \psi
\end{equation}
and 
\begin{equation*}
\epsilon \int_{0}^{T} \int_{\mathbb{T}} \partial_x \tilde{u}_\epsilon \beta^{''}_\delta (\tilde{u}_\epsilon)\partial_x \tilde{u}_\epsilon \psi \le \sqrt{\epsilon} \| \psi \|_{L^\infty} \| \sqrt{\tilde{\rho}} \partial_x \tilde{u}_\epsilon \|^2_{L^2_{t,x}} \| \beta^{''}_\delta \|_{L^\infty} \le \sqrt{\epsilon} c \| \beta^{''}_\delta \|_{L^\infty},
\end{equation*}
\begin{equation*}
\epsilon \int_{0}^{T} \int_{\mathbb{T}} \partial_x \tilde{u}_\epsilon \beta^{'}_\delta (\tilde{u}_\epsilon) \partial_x \psi \le \sqrt{\epsilon} \| \sqrt{\epsilon} \partial_x \tilde{u}_\epsilon \|_{L^2_{t,x}} \| \partial_x \psi \|_{L^2} \| \beta^{'}_\delta (\tilde{u}_\epsilon) \|_{L^\infty}
\end{equation*}
hence we deduce
\begin{equation}
\int_{0}^{T} \int_{\mathbb{T}} \tilde{N}_\epsilon \beta^{'}_\delta (\tilde{u}_\epsilon)\psi dxdt \longrightarrow 0,\quad \tilde{\mathbb{P}}-a.s..
\end{equation}
Concerning the remaining terms with $\beta^{''}$ we observe that the total variation satisfies
$$|\mu_\epsilon^{d,\delta}|\le  \int_{0}^{T} \int_{\mathbb{T}}  \tilde{\rho}_\epsilon^{\frac{\alpha}{2}} | \tilde{M}_\epsilon| \beta^{''}_\delta (\tilde{u}_\epsilon) |\partial_x \tilde{u}_\epsilon| dxdt \le \| \beta_\delta^{''}\|_{\infty} \| \tilde{\rho}_\epsilon^{\frac{\alpha}{2}} \partial_x \tilde{u}_\epsilon \|_{L^2_{t,x}} \| \tilde{M}_\epsilon \|_{L^2_{t,x}},$$
while the It$\text{\'{o}}$ correction term is estimated by using \eqref{f2} and \eqref{beta hp} as follows
\begin{equation}
|\mu^{s,\delta}_\epsilon| \le \int_{0}^{T} \int_{\mathbb{T}} \sum_{k \in \mathbb{N}} \dfrac{1}{2} \tilde{\rho}_\epsilon \beta_\delta^{''}(\tilde{u}_\epsilon) | F_k|^2 dxdt \le C \| \beta^{''}_\delta \|_{\infty} \sum_{k \in \mathbb{N}} \alpha^2_k\le C \delta.
\end{equation}
Now we use Proposition \ref{conv stoch jak} to deduce  $$\langle \mu_\epsilon^{d,\delta}, \psi \rangle \rightarrow \langle \mu^{d,\delta}, \psi \rangle, \quad  \langle \mu_\epsilon^{s,\delta}, \psi \rangle \rightarrow \langle \mu^{s,\delta}, \psi \rangle,\quad \tilde{\mathbb{P}}-a.s.$$ for all $\psi \in C((0,T) \times \mathbb{T}).$ 
\\
\\
In order to pass to the limit in the stochastic integral, we observe that 
$$F_k (\tilde{\rho}_\epsilon, \tilde{u}_\epsilon) \rightarrow F_k(\tilde{\rho},\tilde{u}) \quad \text{in} \; L^p((0,T) \times \mathbb{T}) \quad \text{for all} \quad  p \ge 1, \quad \tilde{\mathbb{P}}-a.s. $$
On the other hand 
\begin{equation*}
\begin{split}
& \mathbb{E} \int_{0}^{T} \bigg\| \int_{\mathbb{T}} \tilde{\rho}_\epsilon \mathbb{F} \beta^{'}_\delta (\tilde{u}_\epsilon) dx \bigg\|^2_{L^2(\mathfrak{U}, \mathbb{R})} dt \le \mathbb{E} \int_{0}^{T} \sum_{k=1}^{\infty} \bigg( \int_{\mathbb{T}}\tilde{\rho}_\epsilon F_k \beta^{'}_\delta (\tilde{u}_\epsilon) dx \bigg)^2ds \\ & \lesssim \sum_{k=1}^{\infty} \alpha^2_k \mathbb{E} \int_{0}^{T} \bigg( \int_{\mathbb{T}} |\tilde{\rho}_\epsilon \tilde{u}_\epsilon|^p dx \bigg)^2 \lesssim \sum_{k=1}^{\infty} \alpha^2_k,
\end{split}
\end{equation*}
hence we get $$  \int_{\mathbb{T}} \tilde{\rho}_\epsilon \mathbb{F}  \beta^{'}_\delta (\tilde{u}_\epsilon) \psi dx \longrightarrow \int_{\mathbb{T}} \tilde{\rho}  \mathbb{F}  \beta^{'}_\delta (\tilde{u}) \psi dx \quad \text{in} \; L^2(0,T;L_2(\mathfrak{U},\mathbb{R})), \;  \tilde{\mathbb{P}}-a.s..$$
Finally, by recalling the convergence result for $\tilde{W}_\epsilon$ in Proposition \ref{conv stoch jak} we use Lemma \ref{Stoch int conv Lemma} in order to get 
\begin{equation}\label{Prop stoch int conv}
\int_{0}^{T} \int_{\mathbb{T}} \tilde{\rho}_\epsilon \mathbb{F}  \beta^{'}_\delta (\tilde{u}_\epsilon) \psi dxd\tilde{W} \longrightarrow \int_{0}^{T} \int_{\mathbb{T}} \tilde{\rho} \mathbb{F}  \beta^{'}_\delta (\tilde{u}) \psi dxd\tilde{W}, \quad \text{in} \; L^2(0,T) \; \text{in probability}.
\end{equation}
\noindent
As a consequence, by sending $\epsilon \rightarrow 0$ in \eqref{truncated weak epsilon} we have that the limit process $(\tilde{\rho},\tilde{ \zeta}, \tilde{\Lambda})$ satisfies
\begin{equation}\label{trunc delta form}
\begin{split}
& \int_{\mathbb{T}} \tilde{\rho}_{0} \beta_\delta(\tilde{u}_{0})\psi(0,x)dx  +\int_{0}^{T} \int_{\mathbb{T}} \tilde{\rho} \beta_\delta(\tilde{u}) \partial_t \psi dxdt+\int_{0}^{T} \int_{\mathbb{T}}\tilde{\rho} \tilde{u} \beta_\delta (\tilde{u})\partial_x \psi dxdt \\ & + \int_{0}^{T} \int_{\mathbb{T}}2 \tilde{\rho}^{\frac{\gamma}{2}}\partial_x \tilde{\rho}^{\frac{\gamma}{2}} \beta^{'}_\delta(\tilde{u}) \psi dxdt -\int_{0}^{T} \int_{\mathbb{T}} \tilde{\rho}^{\frac{\alpha}{2}}\tilde{M} \beta^{'}_\delta (\tilde{u}) \partial_x \psi dxdt+\langle \mu^\delta, \psi \rangle \\ & +  \int_{0}^{T} \int_{\mathbb{T}} \tilde{\rho} \mathbb{F}  \beta^{'}_\delta (\tilde{u}) \psi dx \text{d}\tilde{W}=0,
\end{split}
\end{equation}
for any $\psi \in C^{\infty}_c([0,T) \times \mathbb{T})), \quad \tilde{\mathbb{P}}-a.s.$
\\
\\
Finally we prove that the limit process $[\tilde{\rho},\tilde{m}, \tilde{\zeta},\tilde{\Lambda},\tilde{W}]$ satisfies the energy dissipation equality \eqref{energy dissipation equality}.  We claim that the following identity is satisfied
\begin{equation}\label{delta dissipation equality}
\begin{split}
\int_{0}^{T} \int_{\mathbb{T}} \tilde{\rho}^{\frac{\alpha}{2}} \tilde{\zeta} \beta^{'}_\delta(\tilde{u}) \phi_m(\tilde{\rho}) \psi dxdt & = - \int_{0}^{T} \int_{\mathbb{T}} \tilde{\rho}^{\alpha-\frac{1}{2}} \tilde{\Lambda} \beta^{'}_\delta(\tilde{u}) \phi_m(\tilde{\rho}) \partial_x \psi dxdt  \\ & -2\int_{0}^{T} \int_{\mathbb{T}} \tilde{\rho}^{\frac{\alpha-1}{2}} \tilde{\Lambda} \partial_x \tilde{\rho}^{\frac{\alpha}{2}}\beta^{'}_\delta(\tilde{u}) \phi_m(\tilde{\rho}) \psi dxdt \\ & -  \int_{0}^{T} \int_{\mathbb{T}} \tilde{\rho}^{\alpha-\frac{1}{2}} \tilde{\Lambda} \beta^{'}_\delta(\tilde{u}) \partial_x \tilde{\rho} \phi^{'}_m(\tilde{\rho}) \psi dxdt \\ & -\mathcal{M}_{m,\delta},
\end{split}
\end{equation}
for all test functions $\psi \in C^\infty_c((0,T)) \times \mathbb{T};\mathbb{R}), \;  \tilde{\mathbb{P}}-a.s.$,
\\ \\
where \begin{equation}
\phi_m(y)= \begin{cases}
0 & \;  \;  \; 0 \le y \le \frac{1}{2m}\\
2my-1   \qquad &  \frac{1}{2m} \le y \le \frac{1}{m}\\
1 &  \; \,\frac{1}{m} \le y 
\end{cases}
\end{equation}
and $\mathcal{M}_{m,\delta}$ is the measure such that 
\begin{equation}
\langle  \tilde{\rho}_\epsilon^{\frac{\alpha}{2}} \beta^{''}_\delta (\tilde{u_\epsilon}) \tilde{u_\epsilon} \tilde{\zeta}_\epsilon \phi_m(\tilde{\rho_\epsilon}),\psi \rangle \rightarrow \langle \mathcal{M}_{m,\delta}, \psi \rangle,\quad \tilde{\mathbb{P}}-a.s. 
\end{equation}
as $\epsilon \rightarrow 0,$ for all $\psi \in C((0,T) \times \mathbb{T}). $
\\
\\
The strong solution $(\tilde{\rho}_\epsilon, \tilde{u}_\epsilon)$ indeed satisfies
\begin{equation} 
\begin{split}
\int_{0}^{T} \int_{\mathbb{T}} \tilde{\rho}^{\frac{\alpha}{2}}_\epsilon \tilde{\zeta}_\epsilon \beta^{'}_\delta(\tilde{u}_\epsilon)\phi_m(\tilde{\rho_\epsilon}) \psi dxdt  = & - \int_{0}^{T} \int_{\mathbb{T}} \tilde{\rho}_\epsilon^{\alpha-\frac{1}{2}} \tilde{\Lambda}_\epsilon \beta^{'}_\delta(\tilde{u}_\epsilon) \phi_m(\tilde{\rho_\epsilon})\partial_x\psi dxdt  \\ & -2\int_{0}^{T} \int_{\mathbb{T}} \tilde{\rho}_\epsilon^{\frac{\alpha-1}{2}} \tilde{\Lambda}_\epsilon \partial_x \tilde{\rho}_\epsilon^{\frac{\alpha}{2}}\beta^{'}_\delta(\tilde{u}_\epsilon) \phi_m(\tilde{\rho_\epsilon}) \psi dxdt \\ & - \int_{0}^{T} \int_{\mathbb{T}} \tilde{\rho_\epsilon}^{\alpha-\frac{1}{2}} \tilde{\Lambda_\epsilon} \beta^{'}_\delta(\tilde{u_\epsilon}) \partial_x \tilde{\rho_\epsilon} \phi^{'}_m(\tilde{\rho_\epsilon}) \psi dxdt \\ &  + \int_{0}^{T} \int_{\mathbb{T}} \tilde{\rho}_\epsilon^{\frac{\alpha-1}{2}} \tilde{\Lambda}_\epsilon \beta^{''}_\delta(\tilde{u}_\epsilon) \tilde{\zeta}_\epsilon \phi_m(\tilde{\rho_\epsilon}) \psi dxdt
\end{split}
\end{equation}
for all test functions $\psi \in C^\infty_c((0,T)) \times \mathbb{T};\mathbb{R}), \;  \tilde{\mathbb{P}}-a.s.$ \\
Furthermore, from direct computations the following elementary properties of $\phi_m$ hold 
\begin{equation}\label{phi reg}
\begin{split}
& | y \phi^{'}_m(y)| \le 2, \\ & 
| y^\beta \phi^{'}_m(y)| \le \frac{2}{m^{\beta-1}}, \quad \text{for all}  \;  \beta >1, \\ &
\dfrac{\phi_m(y)}{y^\beta} \le 2^{\beta} m^\beta, \quad \quad \;\; \text{for all}  \;  \beta >0.
\end{split}
\end{equation}
\\
Hence by recalling the regularity estimates \eqref{eps indep reg},  \eqref{beta hp}, Lemma \ref{def u} and the continuity of $\phi_m$ we get
\begin{equation}
\int_{0}^{T} \int_{\mathbb{T}} \tilde{\rho}^{\frac{\alpha}{2}}_\epsilon \tilde{\zeta}_\epsilon \beta^{'}_\delta(\tilde{u}_\epsilon) \phi_m(\tilde{\rho_\epsilon}) \psi dxdt \longrightarrow \int_{0}^{T} \int_{\mathbb{T}} \tilde{\rho}^{\frac{\alpha}{2}} \tilde{\zeta} \beta^{'}_\delta(\tilde{u})\phi_m(\tilde{\rho})  \psi dxdt,  \quad \tilde{\mathbb{P}}-a.s.
\end{equation}
\begin{equation}
\int_{0}^{T} \int_{\mathbb{T}} \tilde{\rho}_\epsilon^{\alpha-\frac{1}{2}} \tilde{\Lambda}_\epsilon \beta^{'}_\delta(\tilde{u}_\epsilon) \phi_m(\tilde{\rho_\epsilon}) \partial_x \psi dxdt \longrightarrow \int_{0}^{T} \int_{\mathbb{T}} \tilde{\rho}^{\alpha-\frac{1}{2}} \tilde{\Lambda} \beta^{'}_\delta(\tilde{u})  \phi_m(\tilde{\rho})\partial_x \psi dxdt, \quad \tilde{\mathbb{P}}-a.s..
\end{equation}
\newline
Moreover since $$\partial_x \tilde{\rho}^{\frac{\alpha}{2}} \tilde{\rho}^{\frac{\alpha-1}{2}}= \dfrac{\alpha}{2(\alpha-\frac{1}{2})} \partial_x \tilde{\rho}^{\alpha-\frac{1}{2}},$$
then 
\begin{equation}
\int_{0}^{T} \int_{\mathbb{T}} \tilde{\rho}_\epsilon^{\frac{\alpha-1}{2}} \tilde{\Lambda}_\epsilon \partial_x \tilde{\rho}_\epsilon^{\frac{\alpha}{2}}\beta^{'}_\delta(\tilde{u}_\epsilon) \phi_m(\tilde{\rho_\epsilon}) \psi dxdt \longrightarrow \int_{\mathbb{T}} \tilde{\rho}^{\frac{\alpha-1}{2}} \tilde{\Lambda} \partial_x \tilde{\rho}^{\frac{\alpha}{2}}\beta^{'}_\delta(\tilde{u}) \phi_m(\tilde{\rho}) \psi dxdt, \quad \tilde{\mathbb{P}}-a.s..
\end{equation}
Concerning the term with $\beta^{''}$ we have that by using Proposition \ref{conv stoch jak}
\begin{equation}
\int_{0}^{T} \int_{\mathbb{T}} \tilde{\rho}_\epsilon^{\frac{\alpha-1}{2}} \tilde{\Lambda}_\epsilon \beta^{''}_\delta(\tilde{u}_\epsilon) \tilde{\zeta}_\epsilon \phi_m(\tilde{\rho_\epsilon})\psi dxdt \longrightarrow \langle \mathcal{M}_{m,\delta}, \psi \rangle,\quad \tilde{\mathbb{P}}-a.s.. 
\end{equation}
Moreover, by virtue of \eqref{beta hp}, \eqref{eps indep reg} and \eqref{phi reg} we have $\tilde{\mathbb{P}}-a.s.$ that 
\begin{equation}
\begin{split}
& \big| \int_{0}^{T} \int_{\mathbb{T}} \tilde{\rho}_\epsilon^{\frac{\alpha-1}{2}} \tilde{\Lambda}_\epsilon \beta^{''}_\delta(\tilde{u}_\epsilon) \tilde{\zeta}_\epsilon \phi_m(\tilde{\rho_\epsilon})\psi dxdt \big| = \big|\int_{0}^{T} \int_{\mathbb{T}} \tilde{\rho}_\epsilon^{\frac{\alpha}{2}-{\frac{1}{4}}} \tilde{\Lambda}_\epsilon \beta^{''}_\delta(\tilde{u}_\epsilon) \tilde{\zeta}_\epsilon \dfrac{\phi_m(\tilde{\rho_\epsilon})}{\tilde{\rho}^{\frac{1}{4}}_\epsilon}\psi dxdt \big| \\ & \le \| \tilde{\rho}_\epsilon^{\frac{\alpha}{2}-{\frac{1}{4}}} \|_{L^\infty} \| \tilde{\Lambda}_\epsilon \|_{L^2} \| \tilde{\zeta}_\epsilon \|_{L^2} \| \beta^{''}(\tilde{u}_\epsilon) \|_{L^\infty} \bigg \| \dfrac{\phi_m(\tilde{\rho_\epsilon})}{\tilde{\rho}^{\frac{1}{4}}_\epsilon} \bigg \|_{L^\infty} \| \psi \|_{L^\infty} \le C \delta m^{\frac{1}{4}}.
\end{split}
\end{equation}
Note that the constant $C$ is independent on $\epsilon$, and therefore the total variation satisfies 
\begin{equation}
| M_{m,\delta} | \le C \delta m^{\frac{1}{4}}\qquad \tilde{\mathbb{P}}-a.s..
\end{equation}
Finally by using Proposition \ref{conv stoch jak} and \eqref{beta hp} we have that $\tilde{\mathbb{P}}-a.s.$
\begin{equation}
\int_{0}^{T} \int_{\mathbb{T}} \tilde{\rho_\epsilon}^{\alpha-\frac{1}{2}} \tilde{\Lambda_\epsilon} \beta^{'}_\delta(\tilde{u_\epsilon}) \partial_x \tilde{\rho_\epsilon} \phi^{'}_m(\tilde{\rho_\epsilon}) \psi dxdt \longrightarrow \int_{0}^{T} \int_{\mathbb{T}} \tilde{\rho}^{\alpha-\frac{1}{2}} \tilde{\Lambda} \beta^{'}_\delta(\tilde{u}) \partial_x \tilde{\rho} \phi^{'}_m(\tilde{\rho}) \psi dxdt=R_{m,\delta}
\end{equation}
and again by virtue of \eqref{phi reg} and \eqref{beta hp}, $\tilde{\mathbb{P}}-a.s$,
\begin{equation}
\begin{split}
& \big| \int_{0}^{T} \int_{\mathbb{T}} \tilde{\rho_\epsilon}^{\alpha-\frac{1}{2}} \tilde{\Lambda_\epsilon} \beta^{'}_\delta(\tilde{u_\epsilon}) \partial_x \tilde{\rho_\epsilon} \phi^{'}_m(\tilde{\rho_\epsilon}) \psi dxdt \big| = \big|  \dfrac{1}{(\alpha-\frac{1}{2})}  \int_{0}^{T} \int_{\mathbb{T}} \partial_x \tilde{\rho_\epsilon}^{\alpha-\frac{1}{2}} \tilde{\Lambda_\epsilon} \beta^{'}_\delta(\tilde{u_\epsilon}) \tilde{\rho_\epsilon} \phi^{'}_m(\tilde{\rho_\epsilon}) \psi dxdt \big| \\ & \le \dfrac{1}{(\alpha-\frac{1}{2})} \| \partial_x \tilde{\rho_\epsilon}^{\alpha-\frac{1}{2}} \|_{L^1} \| \tilde{u}_\epsilon \beta^{'}_\delta(\tilde{u_\epsilon}) \|_{L^\infty} \| \tilde{\rho}_\epsilon^{\frac{3}{2}} \phi^{'}_m(\tilde{\rho_\epsilon}) \|_{L^\infty} \| \psi \|_{L^\infty} \le \dfrac{C}{\delta \sqrt{m}}.
\end{split}
\end{equation}
Hence we deduce 
\begin{equation}
| R_{m,\delta}| \le \dfrac{C}{\delta \sqrt{m}},\qquad \tilde{\mathbb{P}}-a.s
\end{equation}
\noindent
for a given constant $C$ independent on $\epsilon.$
\\
In what follows,  we prove the energy and entropy inequalities \eqref{energy weak def}, \eqref{entropy weak def} satisfied by a weak dissipative martingale solution to system \eqref{stoc quantum} in the sense of Definition \ref{def weak dissip sol}.
Last goal is to perform the asymptotic limit in the total energy balance.  First, we observe that by using Proposition \ref{conv stoch jak} and Theorem \ref{eq in law equation} the energy balance 
\begin{equation}
\begin{split}
&  -\int_{0}^{T} \partial_t \phi  \int_{\mathbb{T}} \bigg[\dfrac{1}{2}\tilde{\rho}_{\epsilon} |\tilde{u}_{\epsilon}|^2+ \dfrac{\tilde{\rho}_{\epsilon}^\gamma}{\gamma-1}+ | \partial_x \sqrt{\tilde{\rho}_{\epsilon}}|^2 \bigg] dxdt +\int_{0}^{T} \phi \int_{\mathbb{T}} (\mu(\tilde{\rho}_{\epsilon}) +\epsilon)| \partial_x \tilde{u}_{\epsilon} |^2 dxdt  \\ & = \phi(0) \int_{\mathbb{T}} \bigg[\dfrac{1}{2}\tilde{\rho}_0|\tilde{u}_0|^2+ \dfrac{\tilde{\rho}_0^\gamma}{\gamma-1}+ | \partial_x \sqrt{\tilde{\rho}_0}|^2 \bigg] dx+\dfrac{1}{2} \int_{0}^{T} \phi \int_{\mathbb{T}} \sum_{k=1}^{\infty} \tilde{\rho}_\epsilon |F_k|^2 dxdt \\ & + \sum_{k=1}^{\infty} \int_{0}^{T} \phi \int_{\mathbb{T}} \tilde{\rho}_\epsilon F_k \tilde{u}_\epsilon dxd\tilde{W}_k , \quad \text{for all} \; \phi \in C_c^\infty ([0,T)),  \; \phi \ge 0, \; \tilde{\mathbb{P}}-a.s.
\end{split}
\end{equation}
is satisfied by $[\tilde{\rho}_\epsilon,  \tilde{\zeta}_\epsilon, \tilde{\Lambda}_\epsilon]$ on the new probability space with initial energy $$\int_{\mathbb{T}} \bigg[ \dfrac{1}{2} \tilde{\rho}_0 |\tilde{u}_0|^2 + \dfrac{\tilde{\rho}^\gamma_0}{\gamma-1}+ \partial_x \sqrt{\tilde{\rho}_0} \bigg] dx. $$
On the other hand,  since $$F_k (\tilde{\rho}_\epsilon, \tilde{u}_\epsilon) \rightarrow F_k(\tilde{\rho},\tilde{u}) \quad \text{in} \; L^p((0,T) \times \mathbb{T}) \quad \text{for all} \quad  p \ge 1, \quad \tilde{\mathbb{P}}-a.s. $$
and $$ \int_{\mathbb{T}} \tilde{\rho}_\epsilon \mathbb{F}  \tilde{u}_\epsilon \psi dx \longrightarrow \int_{\mathbb{T}} \tilde{\rho}  \mathbb{F}  \tilde{u} \psi dx \quad \text{in} \; L^2(0,T;L_2(\mathfrak{U},\mathbb{R})), \;  \tilde{\mathbb{P}}-a.s.,$$
$$\tilde{W}_\epsilon \rightarrow \tilde{W} \in C(0,T;\mathfrak{U}_0), \quad
\tilde{\mathbb{P}}-a.s. $$
we use Lemma \ref{Stoch int conv Lemma} to pass to the limit in the stochastic integral.  Finally by using the uniform bounds \eqref{eps indep reg} we pass to the limit in the It$\hat{\text{o}}$ correction term to get \eqref{energy weak def}.
\\
Similarly,  we observe that by using Proposition \ref{conv stoch jak} and Theorem \ref{eq in law equation} the entropy balance 
\begin{equation}
\begin{split}
&  -\int_{0}^{T} \partial_t \phi  \int_{\mathbb{T}} \bigg[\dfrac{1}{2}\tilde{\rho}_{\epsilon} |\tilde{V}_{\epsilon}|^2+ \dfrac{\tilde{\rho}_{\epsilon}^\gamma}{\gamma-1}+ | \partial_x \sqrt{\tilde{\rho}_{\epsilon}}|^2 \bigg] dxdt + \dfrac{4 \gamma}{(\gamma+\alpha -1)^2} \int_{0}^{T} \phi \int_{\mathbb{T}} | \partial_x \tilde{\rho}_{\epsilon}^{\frac{\gamma+\alpha-1}{2} } |^2 dxdt \\ & + \dfrac{4\epsilon \gamma}{(\gamma-1)^2} \int_{0}^{T} \phi \int_{\mathbb{T}} | \partial_x \tilde{\rho}_{\epsilon}^{\frac{\gamma-1}{2} } |^2 dxdt +\dfrac{\epsilon}{2} \int_{0}^{T} \phi \int_{\mathbb{T}} | \partial_{xx} \log \tilde{\rho}_{\epsilon} |^2 dxdt \\ & + \dfrac{4}{\alpha^2} \int_{0}^{T} \phi \int_{\mathbb{T}} | \partial_{xx} \tilde{\rho}_{\epsilon}^{\frac{\alpha}{2}} |^2 dxdt +\dfrac{4(4-3 \alpha) }{3 \alpha^3}\int_{0}^{T} \phi \int_{\mathbb{T}} \tilde{\rho}_{\epsilon}^{-\alpha} | \partial_x \tilde{\rho}_{\epsilon}^\frac{\alpha}{2}|^4 dxdt  \\ &  =  \phi(0) \int_{\mathbb{T}} \bigg[\dfrac{1}{2}\tilde{\rho}_0 |\tilde{V}_0|^2+ \dfrac{\tilde{\rho}_0^\gamma}{\gamma-1}+ | \partial_x \sqrt{\tilde{\rho}_0}|^2 \bigg] dx+\dfrac{1}{2} \int_{0}^{T} \phi \int_{\mathbb{T}} \sum_{k=1}^{\infty} \tilde{\rho}_\epsilon |F_k|^2 dxdt \\ & + \sum_{k=1}^{\infty} \int_{0}^{T} \phi \int_{\mathbb{T}} \tilde{\rho}_\epsilon F_k \tilde{V}_\epsilon dxdW_k , \quad \text{for all} \; \phi \in C_c^\infty ([0,T)),  \; \phi \ge 0, \; \tilde{\mathbb{P}}-a.s.
\end{split}
\end{equation}
is satisfied by $[\tilde{\rho}_\epsilon,  \tilde{\zeta}_\epsilon, \tilde{\Lambda}_\epsilon]$ on the new probability space with initial entropy $$\int_{\mathbb{T}} \bigg[ \dfrac{1}{2} \tilde{\rho}_0|\tilde{V}_0 |^2 + \dfrac{\tilde{\rho}^\gamma_0}{\gamma-1}+ \partial_x \sqrt{\tilde{\rho}_0} \bigg] dx,$$
with the same lines of argument as before we use \eqref{Prop stoch int conv} and the uniform bounds \eqref{eps indep reg} to pass to the limit in the stochastic integral and the It$\hat{\text{o}}$ correction term to deduce \eqref{entropy weak def}.
\\
\\
\textbf{Step 2: The limit $\delta \rightarrow 0$ and $m\to\infty$}
\\
We perform the limit $\delta \rightarrow 0$ which concludes the proof of the existence result.  We observe that for the total variation of $\mu^\delta= \mu^{d,\delta}+\mu^{s,\delta}$ we have that  $$ | \mu^\delta | (\mathbb{T}) \le \| \beta_\delta^{''} \|_{\infty} \le C\delta.$$
Then by using \eqref{beta hp} and the dominated convergence theorem we have that \eqref{trunc delta form} converges to 
\begin{equation}
\begin{split}
& \int_{\mathbb{T}} \tilde{\rho}_{0} \tilde{u}_{0}\psi(0,x)dx  +\int_{0}^{T} \int_{\mathbb{T}} \tilde{\rho} \tilde{u} \partial_t \psi dxdt+\int_{0}^{T} \int_{\mathbb{T}}\tilde{\rho} \tilde{u}^2 \partial_x \psi dxdt \\ & + \int_{0}^{T} \int_{\mathbb{T}}2 \tilde{\rho}^{\frac{\gamma}{2}}\partial_x \tilde{\rho}^{\frac{\gamma}{2}} \psi dxdt -\int_{0}^{T} \int_{\mathbb{T}} \tilde{\rho}^{\frac{\alpha}{2}}\tilde{M}  \partial_x \psi dxdtt  \\ & +  \int_{0}^{T} \int_{\mathbb{T}} \tilde{\rho} \mathbb{F}  \psi dx \text{d}\tilde{W}=0,
\end{split}
\end{equation}
for any $\psi \in C^{\infty}_c([0,T) \times \mathbb{T})), \quad \tilde{\mathbb{P}}-a.s.$
\\
\\
In order to complete the proof of the existence of a weak dissipative martingale solution to system \eqref{stoc quantum} in the sense of Definition \ref{def weak dissip sol} it remains to prove the dissipation equality \eqref{energy dissipation equality}.  To this purpose we observe that 
\begin{equation}\label{phi conv}
\tilde{\rho}^{\beta} \phi_m (\tilde{\rho}) \rightarrow \tilde{\rho}^\beta, \; \text{for almost every}\; (x,t) \in \mathbb{T}\times (0,T), \; \tilde{\mathbb{P}}-a.s.
\end{equation}
and we choose $\delta= \dfrac{1}{m^\theta}$ with $\theta \in (\frac{1}{4}, \frac{1}{2}),$ we use \eqref{beta hp},\eqref{eps indep reg} and \eqref{phi conv} to deduce
\begin{equation}
\int_{0}^{T} \int_{\mathbb{T}} \tilde{\rho}^{\frac{\alpha}{2}} \tilde{\zeta} \beta^{'}_\delta(\tilde{u}) \phi_m(\tilde{\rho}) \psi dxdt \longrightarrow \int_{0}^{T} \int_{\mathbb{T}} \tilde{\rho}^{\frac{\alpha}{2}} \tilde{\zeta} \psi dxdt, \quad \tilde{\mathbb{P}}-a.s.
\end{equation}
\begin{equation}
\int_{0}^{T} \int_{\mathbb{T}} \tilde{\rho}^{\alpha-\frac{1}{2}} \tilde{\Lambda} \beta^{'}_\delta(\tilde{u}) \phi_m(\tilde{\rho}) \partial_x \psi dxdt \longrightarrow \int_{0}^{T} \int_{\mathbb{T}} \tilde{\rho}^{\alpha-\frac{1}{2}} \tilde{\Lambda}  \partial_x \psi dxdt, \quad \tilde{\mathbb{P}}-a.s.
\end{equation}
\begin{equation}
\int_{\mathbb{T}} \tilde{\rho}^{\frac{\alpha-1}{2}} \tilde{\Lambda} \partial_x \tilde{\rho}^{\frac{\alpha}{2}}\beta^{'}_\delta(\tilde{u}) \phi_m(\tilde{\rho}) \psi dxdt \longrightarrow \int_{\mathbb{T}} \tilde{\rho}^{\frac{\alpha-1}{2}} \tilde{\Lambda} \partial_x \tilde{\rho}^{\frac{\alpha}{2}} \psi dxdt, \quad \tilde{\mathbb{P}}-a.s.
\end{equation}
\begin{equation}
| M_{m,\delta} |  \longrightarrow  0, \quad \tilde{\mathbb{P}}-a.s.
\end{equation}
\begin{equation}
| R_{m,\delta}| \longrightarrow 0, \quad \tilde{\mathbb{P}}-a.s..
\end{equation}
as $m \rightarrow \infty.$ Hence by passing to the limit in \eqref{delta dissipation equality} we get that $(\tilde{\rho}, \tilde{\Lambda}, \tilde{\zeta}),$ satisfies \eqref{energy dissipation equality}. This concludes the convergence result and the proof of Theorem \ref{main theorem weak}.
\\
\\
\subsection*{Acknowledgments}
The authors gratefully acknowledge the partial support by the Gruppo
Na\-zio\-na\-le per l’Analisi Matematica, la Probabilit\`a e le loro
Applicazioni (GNAMPA) of the Istituto Nazionale di Alta Matematica
(INdAM), and by the PRIN 2020 ``Nonlinear evolution PDEs, fluid
dynamics and transport equations: theoretical foundations and
applications'' and by the PRIN2022
``Classical equations of compressible fluids mechanics: existence and
properties of non-classical solutions''. The first and the third authors gratefully acknowledge the partial support by PRIN2022-PNRR ``Some
mathematical approaches to climate change and its impacts.''
\\
\\
\textbf{Declarations}
\\
\\
\textbf{Data Availability.} The authors declare that data sharing is not applicable to this article as no datasets were generated or analysed.
\\
\\
\textbf{Conflict of interest.} The authors declare that they have no conflict of interest.

\end{document}